\documentclass[11pt,a4paper,reqno]{amsart} 
\usepackage[applemac]{inputenc}
\usepackage{amsmath}
\usepackage{amssymb}
\usepackage{euscript}
\usepackage{pdfsync}
\usepackage{mathrsfs}
\usepackage{wasysym}
\usepackage[all]{xy}
\usepackage{color}
\usepackage{rotating}
\usepackage[colorlinks=true,linktocpage=true,pagebackref=false, urlcolor=black, citecolor=black,linkcolor=black]{hyperref}

\newtheorem{mthm}{Theorem}

\newtheorem{thm}{Theorem}[section]
\newtheorem{lem}[thm]{Lemma}
\newtheorem{prop}[thm]{Proposition}

\newtheorem{defpro}[thm]{Definition and Proposition}
\newtheorem{lemdef}[thm]{Proposition and Definition}
\newtheorem{cor}[thm]{Corollary}

\theoremstyle{definition}

\theoremstyle{remark}
\newtheorem{remark}[thm]{Remark}
\newtheorem{remarks}[thm]{Remarks}
\newtheorem{example}[thm]{Example}
\newtheorem{examples}[thm]{Examples}

\newcommand{\fina}{\hspace*{\fill}$\square$}

{\em}

{\em}

\newcounter{substep}
\def\thesubstep{\arabic{substep}}

\newenvironment{substeps}[1]{%
\refstepcounter{substep}\noindent{(\ref{#1}.\thesubstep)\ }\ }%
{\em}

 \newcommand{\R}{{\mathbb R}}
\newcommand{\Q}{{\mathbb Q}}

\newcommand{\sph}{{\mathbb S}}

\newcommand{\gtp}{{\mathfrak p}} \newcommand{\gtq}{{\mathfrak q}}
\newcommand{\gtm}{{\mathfrak m}} \newcommand{\gtn}{{\mathfrak n}}
\newcommand{\gta}{{\mathfrak a}}

\newcommand{\Dd}{{\EuScript D}}
\newcommand{\Zz}{{\EuScript Z}}

\newcommand{\bs}{{\EuScript B}}

\newcommand{\Bb}{{\EuScript B}}

\newcommand{\im}{\operatorname{im}}
\newcommand{\qf}{\operatorname{qf}}

\newcommand{\Int}{\operatorname{Int}}
\newcommand{\dist}{\operatorname{dist}}

\newcommand{\hgt}{\operatorname{ht}}

\newcommand{\Spec}{\operatorname{Spec}}

\newcommand{\ceros}{\operatorname{\mathcal Z}}

\newcommand{\gr}{\operatorname{graph}}
\newcommand{\cl}{\operatorname{Cl}}

\newcommand{\di}{\operatorname{{\mathcal D}}}

\newcommand{\dgt}{\operatorname{d}}

\newcommand{\lc}{\operatorname{lc}}
\newcommand{\diam}{{\text{\tiny$\displaystyle\diamond$}}}
\newcommand{\gtmd}{\operatorname{\gtm^{\diam}\hspace{-1.5mm}}}
\newcommand{\gtnd}{\operatorname{\gtn^{\diam}\hspace{-1.5mm}}}
\newcommand{\Specs}{\operatorname{Spec_s}}
\newcommand{\Speca}{\operatorname{Spec_s^*}}
\newcommand{\Specd}{\operatorname{Spec_s^{\diam}}}
\newcommand{\betas}{\operatorname{\beta_s\!}}
\newcommand{\betaa}{\operatorname{\beta_s^*\!\!}}
\newcommand{\betad}{\operatorname{\beta_s^{\diam}\!\!}}

\newcommand{\x}{{\tt x}} \newcommand{\y}{{\tt y}} 
\newcommand{\z}{{\tt z}} \renewcommand{\t}{{\tt t}}
\newcommand{\s}{{\tt s}}

\newcommand{\veps}{\varepsilon}

\newcommand{\ol }{\overline}

\numberwithin{equation}{section}

\begin{document}

\title[On spectral types of semialgebraic sets]{On spectral types of semialgebraic sets}

\author{Jos\'e F. Fernando}
\author{J.M. Gamboa}
\address{Departamento de \'Algebra, Facultad de Ciencias Matem\'aticas, Universidad Complutense de Madrid, 28040 MADRID (SPAIN)}
\curraddr{}
\email{josefer@mat.ucm.es, jmgamboa@mat.ucm.es}
\thanks{Authors supported by Spanish GR MTM2011-22435}

\subjclass[2010]{Primary 14P10, 54C30; Secondary 12D15, 13E99}
\keywords{Ring of semialgebraic functions, ring of bounded semialgebraic functions, Zariski spectrum, maximal spectrum, semialgebraic homeomorphism, homeomorphism, bricks, local compactness, free maximal ideal associated with a formal path, free maximal ideal associated with a semialgebraic path}

\begin{abstract}
In this work we prove that a semialgebraic set $M\subset\R^m$ is determined (up to a semialgebraic homeomorphism) by its ring ${\mathcal S}(M)$ of (continuous) semialgebraic functions while its ring ${\mathcal S}^*(M)$ of (continuous) bounded semialgebraic functions only determines $M$ besides a distinguished finite subset $\eta(M)\subset M$. In addition it holds that the rings ${\mathcal S}(M)$ and ${\mathcal S}^*(M)$ are isomorphic if and only if $M$ is compact. On the other hand, their respective maximal spectra $\betas M$ and $\betaa M$ endowed with the Zariski topology are always homeomorphic and topologically classify a `large piece' of $M$. The proof of this fact requires a careful analysis of the points of the remainder $\partial M:=\betaa M\setminus M$ associated with formal paths.
\end{abstract}

\maketitle

\section*{Introduction}\label{s1}

A semialgebraic set $M\subset\R^m$ is a (finite) boolean combination of sets defined by polynomial equations and inequalities. A continuous map $f:M\to\R^n$ is \emph{semialgebraic} if its graph is a semialgebraic subset of $\R^{m+n}$; as usual $f$ is a \em semialgebraic function \em if $n=1$. The sum and product of functions defined pointwise endow the set ${\mathcal S}(M)$ of semialgebraic functions on $M$ with a natural structure of a unital commutative ring. In fact ${\mathcal S}(M)$ is an $\R$-algebra and the subset ${\mathcal S}^*(M)$ of bounded semialgebraic functions on $M$ is an $\R$-subalgebra of ${\mathcal S}(M)$. 

It is well known that the rings ${\mathcal S}(M)$ and ${\mathcal S}^*(M)$ are particular cases of the so-called \em real closed rings \em introduced by Schwartz in the 1980s, see \cite{s0}. The theory of real closed rings has been developed in a fruitful attempt to establish new foundations for semi-algebraic geometry with relevant interconnections to model theory, see the results of Cherlin-Dickmann \cite{cd1,cd2}, Schwartz \cite{s0,s1,s2,s3}, Schwartz with Prestel, Madden and Tressl \cite{ps,sm,scht} and Tressl \cite{t0,t1,t2}. We refer the reader to \cite{s1} for a ring theoretic analysis of the concept of a real closed ring. This theory, which vastly generalizes the classical techniques concerning the semialgebraic spaces of Delfs-Knebusch \cite{dk2}, provides a powerful machinery to approach problems concerning certain rings of real valued functions and contributes to achieve a better understanding of the algebraic properties of such rings and the topological properties of their spectra. We highlight some relevant families of real closed rings: (1) real closed fields; (2) rings of real-valued continuous functions on Tychonoff spaces; (3) rings of semi-algebraic functions on semi-algebraic sets; and more generally (4) rings of definable continuous functions on definable sets in o-minimal expansions of fields.

\subsection*{Main results} 
The main purpose of this work is to analyze to what extend the rings ${\mathcal S}(M)$ and ${\mathcal S}^*(M)$ classify the semialgebraic set $M$ up to a semialgebraic homeomorphism, that is, a semialgebraic map that is also a homeomorphism; of course, the inverse of such a map is also a semialgebraic map. Recall that \em homeomorphic \em and \em semialgebraically homeomorphic \em are not the same notion: Shiota and Yokoi presented two compact homeomorphic semialgebraic sets that are not semialgebraically homeomorphic, see \cite{ys}.

We show in Theorem \ref{unbound1} that the mentioned classification holds for ${\mathcal S}(M)$. In Theorem \ref{bound1} we prove that the ring ${\mathcal S}^*(M)$ classifies $M$ up to a semialgebraic homeomorphism and besides the finite set $\eta(M)$ consisting of those points of $M$ having an open neighborhood in $M$ that is semialgebraically homeomorphic to the interval $[0,1)$. In the following $M\subset\R^m$ and $N\subset\R^n$ always denote semialgebraic sets. 

\begin{mthm}[Spectral type]\label{unbound1}
The rings ${\mathcal S}(N)$ and ${\mathcal S}(M)$ are isomorphic if and only if the semialgebraic sets $N$ and $M$ are semialgebraically homeomorphic.
\end{mthm}

\begin{mthm}[Bounded spectral type]\label{bound1}
The rings ${\mathcal S}^*(N)$ and ${\mathcal S}^*(M)$ are isomorphic if and only if the semialgebraic sets $N\setminus\eta(N)$ and $M\setminus\eta(M)$ are semialgebraically homeomorphic.
\end{mthm}

In Theorem \ref{boundunbound1} we analyze how two semialgebraic sets $M$ and $N$ are related if we know that ${\mathcal S}(N)$ and ${\mathcal S}^*(M)$ are isomorphic. 

\begin{mthm}[Comparison]\label{boundunbound1}
The rings ${\mathcal S}(N)$ and ${\mathcal S}^*(M)$ are isomorphic if and only if the semialgebraic set $N$ is compact and the semialgebraic sets $N\setminus\eta(N)$ and $M\setminus\eta(M)$ are semialgebraically homeomorphic.
\end{mthm}

Moreover, the Zariski spectra $\Specs(M)$ of ${\mathcal S}(M)$ and $\Speca(M)$ of ${\mathcal S}^*(M)$ are homeomorphic if and only if the rings ${\mathcal S}(M)$ and ${\mathcal S}^*(M)$ are isomorphic. This happens if and only if $M$ is compact (see Theorem \ref{boundunbound0}) or equivalently if both rings coincide (see Lemma \ref{compact0}). These results explain to what extend the rings ${\mathcal S}(M)$ and ${\mathcal S}^*(M)$ determine the topology of $M$. Our last result deals with maximal ideals instead of prime ideals of these rings. Indeed, the maximal spectra $\betas M$ of ${\mathcal S}(M)$ and $\betaa M$ of ${\mathcal S}^*(M)$, also called the semialgebraic Stone-C\v{e}ch compactification of $M$, are always homeomorphic (see \ref{homeo}) but the involved homeomorphism is not natural from a categorical point of view \cite[3.6]{fg3}. However, $\betaa M$ `almost' topologically classifies $M$. We denote the set of points of $M$, which have a compact neighborhood in $M$, with $M_{\lc}$. 

\begin{mthm}[Maximal spectral type]\label{betabound}
Let $\gamma:\betaa N\to\betaa M$ be a homeomorphism. Then the restriction map $\gamma\,|:N_{\lc}\setminus\eta(N_{\lc})\to M_{\lc}\setminus\eta(M_{\lc})$ is also a homeomorphism.
\end{mthm}

To prove this fact, a reasonable strategy would be to search for a topological condition that characterizes the points of $M$ among those of $\betaa M$. In \cite[9.6-7]{gj} the authors prove that if $X$ is a metrizable space, then $X$ is the set of $G_{\delta}$-points of the Stone--\v{C}ech compactification $\beta X$ of $X$. Thus, it would seem reasonable to follow a similar strategy. As we show in Lemma \ref{cn}, all points of $M$ have a countable basis of neighborhoods in $\betaa M$. However, as we prove in Proposition \ref{gpn3}, the same happens for a large subset of the remainder $\partial M:=\betaa M\setminus M$ constituted by free maximal ideals associated with formal paths (see Section \ref{s4}). In our setting the clue property is to `admit a metrizable neighborhood in $\betaa M$'. Referring to this, we characterized the semialgebraic sets $M$ whose maximal spectrum $\betaa M$ is a metrizable space in \cite[5.17]{fg5}: this happens for those semialgebraic sets whose maximal spectrum $\betaa M$ is homeomorphic to a semialgebraic set. In Lemma \ref{mn} we determine the set of points of $\betaa M$ that admit a metrizable neighborhood in $\betaa M$.

\subsection*{Structure of the article} 
In Section \ref{s2} we compile the preliminary terminology and results concerning Zariski and maximal spectra of rings of semialgebraic and bounded semialgebraic functions that we use along this work. Most of the results in Section \ref{s2} are collected from \cite{fe1,fg2,fg3,fg5} and presented without proofs. The reading can be started directly in Section \ref{s3} and referred to the preliminaries only when needed. Section \ref{s3} is devoted to the study of homeomorphisms between Zariski spectra of rings of semialgebraic and bounded semialgebraic functions on a semialgebraic set (see Theorems \ref{unbound0}, \ref{bound0} and \ref{boundunbound0}) and to obtain Theorems \ref{unbound1}, \ref{bound1} and \ref{boundunbound1} stated above as a byproduct. In Section \ref{s4} we study a large family of points of the remainder associated with formal paths; these points have a countable basis of neighborhoods in $\betaa M$. Finally, Section \ref{s5} is dedicated to the analysis of homeomorphisms between maximal spectra of rings of semialgebraic (and bounded semialgebraic) functions on a semialgebraic set as well as to prove Theorem \ref{betabound}.

\section{Preliminaries on spectra of rings of semialgebraic functions}\label{s2}

In the following $M\subset\R^m$ denotes a semialgebraic set. Some statements and results are simultaneously valid for ${\mathcal S}(M)$ and ${\mathcal S}^*(M)$. In such a case and to avoid unnecessary repetitions, we denote both rings with ${\mathcal S}^{\diam}(M)$. For each function $f\in{\mathcal S}^{\diam}(M)$ and each semialgebraic subset $N\subset M$ we denote $Z_N(f):=\{x\in N:\, f(x)=0\}$ and $D_N(f):=N\setminus Z_N(f)$. If $N=M$, we say that $Z_M(f)$ is the \em zero set \em of $f$. We denote the open ball of $\R^m$ with center $x$ and radius $\veps>0$ with $\Bb(x,\veps)$ and the corresponding closed ball with $\ol{\Bb}(x,\veps)$. Sometimes it will be useful to assume that the semialgebraic set $M$ we are working with is bounded. Such an assumption can be done without loss of generality because the semialgebraic homeomorphism
$$
h:\Bb(0,1)\to\R^m,\ x\mapsto\frac{x}{\sqrt{1-\|x\|^2}}
$$
induces a ring isomorphism ${\mathcal S}(M)\to {\mathcal S}(h^{-1}(M)),\,f\mapsto f\circ h$ that maps ${\mathcal S}^*(M)$ onto ${\mathcal S}^*(h^{-1}(M))$.

\subsection{Bricks of a semialgebraic set}
Recall the following decomposition of $M$ as an irredundant finite union of closed pure dimensional semialgebraic subsets of $M$ as well as some of its main properties \cite[3.2-3]{fe1}. 

\begin{lemdef}\label{bricks}
There exists a (unique) finite family $\{M_1,\ldots,M_r\}$ of semialgebraic subsets of $M$ satisfying the following properties: 
\begin{itemize}
\item[(i)] Each $M_i$ is the closure of the set of points in $M$ whose local dimension is equal to some fixed value. In particular, $M_i$ is pure dimensional and closed in $M$. 
\item[(ii)] $M=\bigcup_{i=1}^rM_i$.
\item[(iii)] $M_i\setminus\bigcup_{j\neq i}M_j$ is dense in $M_i$.
\item[(iv)] $\dim(M_i)>\dim(M_{i+1})$ for $i=1,\ldots,r-1$. In particular, $\dim(M_1)=\dim(M)$.
\end{itemize}

\em We call the sets $M_i$ the \em bricks \em of $M$ and denote the \em family of bricks of $M$ \em with $\bs_M:=\{\bs_i(M):=M_i\}_{i=1}^r$.
\end{lemdef}

\begin{cor}\label{closuremi}
Let $X\subset\R^n$ be a semialgebraic set such that $M$ is dense in $X$. Then the families $\bs_M$ and $\bs_X$ of bricks of $M$ and $X$ satisfy the following relations:
\begin{itemize}
\item[(i)] $\bs_X:=\{\bs_i(X)=\cl_X(\bs_i(M))\}_i$,
\item[(ii)] $\bs_M:=\{\bs_i(M)=\bs_i(X)\cap M\}_i$.
\end{itemize}
\end{cor}

\subsection{Locally closed semialgebraic sets}
Local closedness has been revealed as an important property for the validity of results, which are in the core of semialgebraic geometry. For instance, it is proved in \cite{fg1} that the classical \L ojasiewicz inequality and the Nullstellensatz for ${\mathcal S}(M)$ work if and only if $M$ is locally closed. It also contains versions of the previous results for ${\mathcal S}^*(M)$. The locally closed subsets of a locally compact topological space coincide with the locally compact ones \cite[\S9.7. Prop.12-13]{bo}. The sets $\cl_{\R^n}(M)$ and $U:=\R^n\setminus(\cl_{\R^n}(M)\setminus M)$ are semialgebraic. If $M$ is locally compact, then $U$ is open in $\R^n$ and $M$ is the intersection of a closed and an open semialgebraic subset of $\R^n$. Let us recall some of the main properties of the largest locally compact and dense subset $M_{\lc}$ of a semialgebraic set $M$. Its construction is the main goal of \cite[9.14-9.21]{dk2}.

\begin{prop}\label{rho} 
Define $\rho_0(M):=\cl_{\R^n}(M)\setminus M$ and 
$$
\rho_1(M):=\rho_0(\rho_0(M))=\cl_{\R^n}(\rho_0(M))\cap M. 
$$
Then the semialgebraic set $M_{\lc}:=M\setminus\rho_1(M)=\cl_{\R^n}(M)\setminus\cl_{\R^n}(\rho_0(M))$ is the largest locally compact and dense subset of $M$ and coincides with the set of points of $M$ that have a compact neighborhood in $M$.
\end{prop}

\begin{remarks}\label{curvas}
(i) If $M$ has dimension $\leq 1$, then $M$ is locally compact. This assertion, which follows from Proposition \ref{rho} and the fact that $\dim(\rho_1(M))<\dim(\rho_0(M))<\dim(M)$ (see \cite[2.8.13]{bcr}), is a well-known fact that appears in \cite{br}.

(ii) Let us denote the subset of points $p\in M$ of local dimension $\dim_p(M)\geq 2$ with $M^{\geq2}$; we refer the reader to \cite[2.8.10-11]{bcr} for further details about the local dimension of semialgebraic sets. This set $M^{\geq2}$ is semialgebraic because it is the union of all bricks of $M$ of dimension $\geq 2$. Moreover, $M=M^{\geq2}\cup L$ where $L\subset M$ is the union of bricks of $M$ of dimension $\leq1$, so $L$ is the closure of the set of points of $M$ of local dimension $\leq1$ in $M$. Observe also that $M^{\geq2}\cap L$ is either empty or a finite set.

(iii) More generally, if the set $M^{\geq2}$ of points of local dimension $\geq 2$ is compact, then $M$ is locally compact. Indeed, with the notations in (ii) and if $M^{\geq2}$ is compact, then 
$$
\rho_0(M)=\cl_{\R^n}(M)\setminus M=\cl_{\R^n}(M^{\geq2}\cup L)\setminus(M^{\geq2}\cup L)=\cl_{\R^n}(L)\setminus(M^{\geq2}\cup L)=\rho_0(L)\setminus M^{\geq2}
$$ 
is a semialgebraic set of dimension $0$ (see Proposition \ref{bricks}), hence a finite set, so $\rho_1(M)$ is empty. Thus, $M=M_{\lc}$ is locally compact.
\end{remarks}

\subsection{Zariski spectra of rings of semialgebraic functions.}\label{zr}
We present some results concerning the Zariski spectra of rings of semialgebraic functions and bounded semialgebraic functions on a semialgebraic set \cite[\S3-\S6]{fg3}. The \em Zariski spectrum \em $\Specd(M):=\Spec({\mathcal S}^{\diam}(M))$ of ${\mathcal S}^{\diam}(M)$ is the collection of all prime ideals of ${\mathcal S}^{\diam}(M)$. This set $\Specd(M)$ is usually endowed with the Zariski topology, which has the family of sets $\Dd_{\Specd(M)}(f):=\{\gtp\in\Specd(M):\, f\not\in\gtp\}$ as a basis of open sets and where $f\in{\mathcal S}^{\diam}(M)$. We denote $\Zz_{\Specd(M)}(f):=\Specd(M)\setminus\Dd_{\Specd(M)}(f)$. 

If $p\in M$, we denote the maximal ideal of all functions in ${\mathcal S}^{\diam}(M)$ vanishing at $p$ with $\gtmd_p$. Observe that the map $\phi:M\to\Specd(M),\ p\mapsto\gtmd_p$ embeds $M$ endowed with the Euclidean topology into $\Specd(M)$ as a dense subspace. 

\subsubsection{}\label{closedspec}
Given a semialgebraic map $\varphi:N\to M$, there exists a unique continuous map $\Specd(\varphi):\Specd(N)\to\Specd(M)$, which extends $\varphi$. In fact, if $N\subset M$ and $N$ is closed in $M$, then $\Specd(N)\cong\cl_{\Specd(M)}(N)$ via $\Specd({\tt j})$ where ${\tt j}:N\hookrightarrow M$ is the inclusion map \cite[4.6]{fg3}.

\subsection{Semialgebraic depth.}\label{ht}
Let us recall the concept of semialgebraic depth of a prime ideal introduced and developed in \cite[\S2]{fg2} and \cite[\S4]{fe1}. This invariant is useful to estimate the coheight of a prime ideal in another prime ideal. 

Let $\gtp\subset\gtq$ be two prime ideals of ${\mathcal S}^{\diam}(M)$. The \em coheight of $\gtp$ in $\gtq$ \em is the maximum of the integers $r\geq0$ such that there exists a chain of prime ideals $\gtp:=\gtp_0\subsetneq\cdots\subsetneq\gtp_r=:\gtq$. We define the \em coheight of a prime ideal $\gtp$ of ${\mathcal S}^{\diam}(M)$ \em as the coheight of $\gtp$ in the unique maximal ideal of ${\mathcal S}^{\diam}(M)$ containing $\gtp$. In particular, the height of a maximal ideal $\gtm$ of ${\mathcal S}^{\diam}(M)$ is the maximum of the coheights of the minimal prime ideals of ${\mathcal S}^{\diam}(M)$ contained in $\gtm$. 

An ideal $\gta$ of ${\mathcal S}(M)$ is a $z$-ideal if, whenever $f,g\in {\mathcal S}(M)$ satisfy $Z_M(f)\subset Z_M(g)$ and $f\in\gta$, then $g\in\gta$. Moreover, if $M$ is locally compact, then all prime ideals of ${\mathcal S}(M)$ are $z$-ideals as a straightforward consequence of \cite[2.6.6]{bcr}.

The \em semialgebraic depth \em of a prime ideal $\gtp$ of ${\mathcal S}(M)$ is $\dgt_M(\gtp):=\min\{\dim Z_M(f):\, f\in\gtp\}$. In \cite[4.14(i)]{fe1} it is proved that if $\gtp\subset\gtq$ are two prime $z$-ideals of ${\mathcal S}(M)$, then the coheight of $\gtp$ in $\gtq$ is $\leq\dgt_M(\gtp)-\dgt_M(\gtq)$. 

\begin{defpro}\label{seteta}
\em Let $\eta(M)$ be the set of points of $M$ that have an open neighborhood in $M$ that is semialgebraically homeomorphic to the interval $[0,1)$. \em Then 
\begin{itemize}
\item[(i)] The set $\eta(M)$ is finite.
\item[(ii)] For each point $p\in \eta(M)$, the maximal ideal $\gtm_p^*$ of ${\mathcal S}^*(M)$ corresponding to $p$ contains properly just one prime ideal of ${\mathcal S}^*(M)$.
\item[(iii)] ${\mathcal S}^*(M)\cong{\mathcal S}^*(M\setminus\eta(M))$.
\end{itemize}
\end{defpro}
\begin{proof}
(i) First observe that if none of the bricks of $M$ has dimension one, then $\eta(M)=\varnothing$ and there is nothing to prove. Otherwise let $\bs_i(M)$ be the only $1$-dimensional brick of $M$. By \cite[2.9.10]{bcr} the set $\bs_i(M)$ is the disjoint union of a finite number of Nash submanifolds $N_j$ and each of them is either Nash diffeomorphic to the open interval $(0,1)$ or to a point. Clearly, $\eta(M)$ is contained in the union of those $N_j$'s, which are points. Hence, $\eta(M)$ is a finite set. 

(ii) Let $B$ be a compact neighborhood of $p$ in $\R^n$ such that $Z:=M\cap B$ is semialgebraically homeomorphic to $[0,1]$. Let $T:=\cl_M(M\setminus Z)$ and note that $p\not\in T$ and $M=T\cup Z$. Then $\Speca(M)=\cl_{\Speca(M)}(T)\cup\cl_{\Speca(M)}(Z)$. As $T$ is closed in $M$, we have $p\equiv\gtm_p^*\in\Speca(M)\setminus\cl_{\Speca(M)}(T)=\cl_{\Speca(M)}(Z)\setminus\cl_{\Speca(M)}(T)$. 

Next, $\cl_{\Speca(M)}(Z)$ is by \ref{closedspec} homeomorphic to $\Speca(Z)\cong\Speca([0,1])$. Under this homeomorphism we may assume that $p$ corresponds to the maximal ideal $\gtm_1$ in $\Speca([0,1])$. Since $I:=[0,1]$ is locally compact, we know by \ref{ht} that if $\gtp$ is a prime ideal of ${\mathcal S}(I)$ (properly) contained in $\gtm_1$, then $0=\dgt_I(\gtm_1)<\dgt_I(\gtp)\leq1$. Thus, $\dgt_I(\gtp)=1$ and $\gtp$ is by \cite[4.5]{fe1} a minimal prime ideal contained in $\gtm_1$. Since $1$ is an endpoint of the interval $[0,1]$, we deduce that $\gtp$ is the unique prime ideal properly contained in $\gtm_1$. Therefore $\gtm_p^*$ contains just one prime ideal of ${\mathcal S}(M)$.

(iii) By \cite[2.9]{fg3} the homomorphism $\phi:{\mathcal S}^*(M)\to{\mathcal S}^*(M\setminus\eta(M)),\,f\mapsto f|_{M\setminus\eta(M)}$ is surjective. But it is also injective because $M\setminus\eta(M)$ is dense in $M$.
\end{proof}

\subsection{Maximal spectra of rings of semialgebraic functions}\label{spectracomp}
We focus our attention on a relevant subspace of $\Specd(M)$: its \em maximal spectrum\em. We denote the collection of all maximal ideals of ${\mathcal S}^{\diam}(M)$ with $\betad M$ and consider in $\betad M$ the topology induced by the Zariski topology of $\Specd(M)$. Given $f,f_1,\ldots,f_r\in{\mathcal S}^{\diam}(M)$, we denote in the following
$$
\begin{array}{rcl}
{\mathcal D}_{\betad M}(f)\hspace{-2.2mm}&:=\hspace{-2.2mm}&\Dd_{\Specd(M)}(f)\cap\betad M,\\[4pt]
{\mathcal Z}_{\betad M}(f)\hspace{-2.2mm}&:=\hspace{-2.2mm}&\betad M\setminus{\mathcal D}_{\betad M}(f)=\Zz_{\Specd(M)}(f)\cap\betad M.
\end{array}
$$
By \cite[7.1.25(ii)]{bcr} $\betad M$ is a Hausdorff compactification of $M$. Moreover, as it happens for rings of continuous functions \cite[\S7]{gj}, the respective maximal spectra $\betas M$ and $\betaa M$ of ${\mathcal S}(M)$ and ${\mathcal S}^*(M)$ are homeomorphic \cite[3.5]{fg5}. More precisely,

\subsubsection{}\label{homeo}
The map $\Phi:\betas M\to\betaa M$ that associates with each maximal ideal $\gtm$ of ${\mathcal S}(M)$ the unique maximal ideal $\gtm^*$ of ${\mathcal S}^*(M)$ that contains the prime ideal $\gtm\cap{\mathcal S}^*(M)$ is a homeomorphism. In particular, $\Phi(\gtm_p)=\gtm_p^*$ for all $p\in M$.

Thus, we denote the maximal ideals of ${\mathcal S}^*(M)$ with $\gtm^*$ where $\gtm$ is the unique maximal ideal of ${\mathcal S}(M)$ such that $\gtm\cap{\mathcal S}^*(M)\subset\gtm^*$.

\subsubsection{}\label{cocr} 
The inclusion map $\R\hookrightarrow{\mathcal S}^*(M)/\gtm^*,\, r\mapsto r+\gtm^*$ is an isomorphism of ordered fields because ${\mathcal S}^*(M)/\gtm^*$ is an Archimedean extension of $\R$. Since $\R$ admits a unique automorphism, there is no ambiguity to refer to $f+\gtm^*$ as a real number for every $f\in{\mathcal S}^*(M)$. In particular, the isomorphism ${\mathcal S}^*(M)/\gtm_p^*\cong\R$ identifies $f+\gtm_p^*$ with $f(p)$ for all $p\in M$. Therefore each bounded semialgebraic function $f:M\to\R$ defines a (unique) natural extension $\widehat{f}:\betaa M\to\R,\, \gtm^*\to f+\gtm^*$, which is continuous because given real numbers $a<b$, we have $\widehat{f}^{-1}((a, b))={\mathcal D}_{\betaa M}(f-a+|f-a|,b-f+|b-f|)$.

\subsubsection{}\label{closedbeta1}
By \cite[5.9]{fg3} we know that if $\varphi:N\to M$ is a semialgebraic map between semialgebraic sets $N$ and $M$, then $\Speca(\varphi):\Speca(N)\to\Speca(M)$ maps $\betaa N$ into $\betaa M$; we denote the restriction of $\Speca(\varphi)$ to $\betaa N$ with $\betaa \varphi:\betaa N\to\betaa M$. Moreover, in \cite[6.3\&6.5]{fg3} we provide proofs of the following properties. Let $C,C_1,C_2$ be closed semialgebraic subsets of the semialgebraic set $M$ and ${\tt j}:C\hookrightarrow M$ the inclusion map. Then 
\begin{itemize}
\item[(i)] The space $\betaa\, C$ is homeomorphic to $\cl_{\betaa M}(C)\subset\betaa M$ via $\betaa\,{\tt j}:\betaa C\to\betaa M$. 
\item[(ii)] $\cl_{\betaa M}(C_1\cap C_2)=\cl_{\betaa M}(C_1)\cap\cl_{\betaa M}(C_2)$.
\end{itemize}

\subsubsection{}\label{freefixed}
In contrast to ideals of polynomial rings, the zero set of a prime ideal $\gtp$ of ${\mathcal S}^{\diam}(M)$ provides no substantial information about $\gtp$ because it is either a singleton or the empty set. An ideal $\gta$ of ${\mathcal S}^{\diam}(M)$ is said to be \em fixed \em if all functions in $\gta$ vanish simultaneously at some point of $M$. Otherwise the ideal $\gta$ is \em free\em. The fixed maximal ideals of the ring ${\mathcal S}^{\diam}(M)$ are those of the form $\gtm_p^{\diam}$ where $p\in M$. Of course different points $p,q\in M$ define different maximal ideals $\gtm_p^{\diam}$ and $\gtm_q^{\diam}$ in ${\mathcal S}^{\diam}(M)$. Clearly $\gtm_p\cap{\mathcal S}^*(M)=\gtm_p^*$ for each point $p\in M$. In fact, the equality $\gtm\cap{\mathcal S}^*(M)=\gtm^*$ characterizes the fixed maxi\-mal ideals of ${\mathcal S}^{\diam}(M)$ (see \cite[3.7]{fg5}). Namely,
$$
\text{$\gtm^*$ is a fixed ideal $\iff$ $\gtm$ is a fixed ideal $\iff$ $\gtm\cap{\mathcal S}^*(M)=\gtm^*$ $\iff$ $\hgt(\gtm)=\hgt(\gtm^*)$}.
$$

As a straightforward consequence of the previous fact, we get an algebraic characterization of the compactness of a semialgebraic set. 

\begin{lem}\label{compact0}
The following assertions are equivalent:
\begin{itemize}
\item[(i)] $M$ is compact.
\item[(ii)] Each maximal ideal of ${\mathcal S}(M)$ is fixed.
\item[(iii)] Each maximal ideal of ${\mathcal S}^*(M)$ is fixed.
\item[(iv)] ${\mathcal S}(M)={\mathcal S}^*(M)$.
\end{itemize}
\end{lem}
\begin{proof}
The equivalence of (ii) and (iii) has already been commented. Let us check the equivalence of (i) and (ii). Consider the embedding $\phi:M\mapsto\betas M,\ p \mapsto\gtm_p$ and recall that $M$ is dense in $\betas M$. Thus, if $M$ is compact, $\phi(M)=\betas M$ and so all maximal ideals of ${\mathcal S}(M)$ are fixed. Conversely, if the maximal ideals of ${\mathcal S}(M)$ are fixed, then $\phi(M)=\betas M$ is compact and $M$ is compact, too.

Finally we show that (i) and (iv) are equivalent. The identity ${\mathcal S}(M)={\mathcal S}^*(M)$ is obvious if $M$ is compact. Conversely, if ${\mathcal S}(M)={\mathcal S}^*(M)$, then $\gtm\cap{\mathcal S}^*(M)=\gtm^*$ for each maximal ideal $\gtm$ of ${\mathcal S}(M)$. Thus, all maximal ideals of ${\mathcal S}(M)$ are fixed and $M$ is compact by the equivalence of (i) and (ii).
\end{proof}

\section{Homeomorphisms between Zariski spectra}\label{s3}

In this section we prove Theorems \ref{unbound1}, \ref{bound1} and \ref{boundunbound1} stated in the Introduction. We begin with a preliminary crucial result, which has an interest on its own.

\begin{thm}\label{unbound0}
Let $\gamma:\Specs(N)\to\Specs(M)$ be a homeomorphism. Then $\gamma|_N:N\to M$ is a homeomorphism.
\end{thm}
\begin{proof}
Of course, the problem is diminished to prove $\gamma (N)=M$. To that end it is sufficient to check that the families of bricks $\bs_N:=\{\bs_i(N)\}_{i=1}^r$ and $\bs_M:=\{\bs_i(M)\}_{i=1}^s$
of $N$ and $M$ satisfy: 

\vspace{1mm}
\begin{substeps}{unbound0}\label{bab} 
\em $r=s$ and $\gamma(\bs_i(N))=\bs_i(M)$ for $i=1,\ldots,r$.
\end{substeps}

\vspace{1mm}
We begin by proving

\vspace{1mm}
\begin{substeps}{unbound0}\label{babs} 
\em $r=s$, $\dim(\bs_i(N))=\dim(\bs_i(M))$ and for each index $i=1,\dots,r$
$$
\gamma(\cl_{\Specs(N)}(\bs_i(N)))=\cl_{\Specs(M)}(\bs_i(M)).
$$
\end{substeps}\setcounter{substep}{0}
We proceed by induction on the dimension of $N$. Indeed, if $N$ has dimension $0$, then 
$$
\begin{array}{l}
\Specs(N)=N=\bs_1(N)=\cl_{\Specs(N)}(\bs_1(N))\\[4pt]
\Specs(M)=M=\bs_1(M)=\cl_{\Specs(M)}(\bs_1(M))
\end{array}
$$
and there is nothing to prove. Suppose that the result is true if $N$ has dimension $\leq d-1$ and let us see that it also holds for $\dim(N)=d$.

If $\gtp_1\subsetneq\gtp_2$ in ${\mathcal S}(N)$, then $\gamma(\gtp_1)\subsetneq\gamma(\gtp_2)$ in ${\mathcal S}(M)$ as $\gamma$ is a homeomorphism; hence, $\hgt(\gtp)=\hgt(\gamma(\gtp))$ for all $\gtp\in\Specs(N)$. Moreover, if $\gtn$ is a maximal ideal of ${\mathcal S}(N)$, then $\gamma(\gtn)$ is a maximal ideal of ${\mathcal S}(M)$ of its same height. Thus, by \cite[Thm.1, p.2]{fg2}
\begin{multline}\label{bullet1}
\dim(\bs_1(N))=\dim(N)=\dim({\mathcal S}(N))=\max\{\hgt(\gtn):\,\gtn\in\betas N\}\\
=\max\{\hgt(\gtm):\,\gtm\in\betas M\}=\dim({\mathcal S}(M))=\dim(M)=\dim(\bs_1(M)).
\end{multline}

By \ref{closedspec} the spaces $\cl_{\Specs(M)}(\bs_i(M))$ and $\Specs(\bs_i(M))$ are homeomorphic for each index $i$. By \cite[Thm.1, p.2]{fg2} $\dim({\mathcal S}(\bs_i(M)))=\dim(\bs_i(M))$ and by \cite[Thm.2, p.2]{fg2} the height of the maximal ideal $\gtn_p=\{f\in{\mathcal S}(N):\,f(p)=0\}$ equals $\dim(\bs_1(N))=\dim(N)$ for each $p\in\bs_1(N)$. The same happens for all $q\in\bs_1(M)$. On the other hand, the semialgebraic sets $T:=\bigcup_{i=2}^r\bs_i(N)$ and $S:=\bigcup_{j=2}^s\bs_j(M)$, which are respectively closed in $N$ and $M$, have dimension $<d$. Thus, by \cite[Thm.1, p.2]{fg2} all prime ideals in $\Specs(T)$ or $\Specs(S)$ have height $<d$. By \ref{closedspec} we have
$$
\begin{array}{l}
\Specs(T)\cong\cl_{\Specs(N)}(T)=\bigcup_{i=2}^r\cl_{\Specs(N)}(\bs_i(N))\\[4pt] 
\Specs(S)\cong\cl_{\Specs(M)}(S)=\bigcup_{j=2}^s\cl_{\Specs(M)}(\bs_j(M))
\end{array}
$$
and so all prime ideals in each of these sets have also height $<d$. Since $\gamma$ preserves heights, we have $\gamma(\bs_1(N))\subset\cl_{\Specs(M)}(\bs_1(M))$. Thus, as $\gamma$ is a continuous map, $\gamma(\cl_{\Specs(N)}(\bs_1(N)))\subset\cl_{\Specs(M)}(\bs_1(M))$. Since the inverse map of $\gamma$ is also continuous, we conclude by symmetry 
\begin{equation}\label{ast1}
\gamma(\cl_{\Specs(N)}(\bs_1(N)))=\cl_{\Specs(M)}(\bs_1(M)).
\end{equation}
Let use see now
$$
\gamma\Big(\bigcup_{i=2}^r\cl_{\Specs(N)}(\bs_i(N))\Big)=\bigcup_{j=2}^s\cl_{\Specs(M)}(\bs_j(M)).
$$
As $\bs_1(N)$ is closed in $N$, notice that
$$
\bs_i(N)\setminus\cl_{\Specs(N)}(\bs_1(N))
=\bs_i(N)\setminus(\cl_{\Specs(N)}(\bs_1(N))\cap N)=\bs_i(N)\setminus \bs_1(N)
$$
for $i\geq 2$. Since $\bs_i(N)\setminus \bs_1(N)$ is dense in $\bs_i(N)$, we get
\begin{equation}\label{ast2}
\begin{split}
\bigcup_{i=2}^r\cl_{\Specs(N)}(\bs_i(N))&=\cl_{\Specs(N)}\Big(\bigcup_{i=2}^r\bs_i(N)\Big)=\cl_{\Specs(N)}\Big(\bigcup_{i=2}^r\bs_i(N)\setminus \bs_1(N)\Big)\\
&=\cl_{\Specs(N)}\Big(\bigcup_{i=2}^r\bs_i(N)\setminus \cl_{\Specs(N)}(\bs_1(N))\Big).
\end{split}
\end{equation}
By \eqref{ast1} and as $\gamma$ is bijective, we have
\begin{multline*}
\gamma\Big(\bigcup_{i=2}^r\bs_i(N)\setminus\cl_{\Specs(N)}(\bs_1(N))\Big)\subset\Specs(M)\setminus\cl_{\Specs(M)}(\bs_1(M))\\
=\bigcup_{j=2}^s\cl_{\Specs(M)}(\bs_j(M))\setminus\cl_{\Specs(M)}(\bs_1(M))\subset\bigcup_{j=2}^s\cl_{\Specs(M)}(\bs_j(M)).
\end{multline*}
Taking closures, using \eqref{ast2} and the fact that $\gamma$ is a homeomorphism, we conclude
$$
\gamma(\bigcup_{i=2}^r\cl_{\Specs(N)}(\bs_i(N)))\subset\bigcup_{j=2}^s\cl_{\Specs(M)}(\bs_j(M)).
$$ 
Of course, by symmetry we deduce the converse inclusion and so
$$
\gamma\Big(\bigcup_{i=2}^r\cl_{\Specs(N)}(\bs_i(N))\Big)=\bigcup_{j=2}^s\cl_{\Specs(M)}(\bs_j(M)).
$$

Applying the inductive hypothesis to the semialgebraic sets $N':=\bigcup_{i=2}^r\bs_i(N)$ and $M':=\bigcup_{j=2}^s\bs_j(M)$ of dimension $\leq d-1$, we de\-duce that $r=s$, $\dim(\bs_i(N))=\dim(\bs_i(M))$ and 
$\gamma(\cl_{\Specs(N)}(\bs_i(N)))=\cl_{\Specs(M)}(\bs_i(M)))$ for $i=2, \ldots, r$. This in combination with \eqref{bullet1} and \eqref{ast1} proves claim \ref{unbound0}.\ref{babs}.

Now we are ready to prove \ref{unbound0}.\ref{bab}. By \ref{closedspec} we have 
$$
\cl_{\Specs(N)}(\bs_i(N))=\Specs(\bs_i(N))\quad\text{and}\quad\cl_{\Specs(M)}(\bs_i(M))=\Specs(\bs_i(M)). 
$$
Let $d_i:=\dim(\bs_i(N))=\dim(\bs_i(M))$ and denote 
$$
\gamma_i:=\gamma|_{\Specs(\bs_i(N))}:\Specs(\bs_i(N))\to \Specs(\bs_i(M)).
$$
Given a point $\gtn_p\equiv p\in \bs_i(N)$, we claim $\hgt(\gamma_i(\gtn_p))=\hgt(\gtn_p)=d_i$ (see \cite[Thm.2, p.2]{fg2}) and so $\gamma_i(\gtn_p)\in\bs_i(M)$. Otherwise $\gamma_i(\gtn_p)$ would be a free maximal ideal of ${\mathcal S}(\bs_i(M))$ by \ref{freefixed} whose height satisfies
$$
\hgt(\gamma_i(\gtn_p))<\hgt(\gamma_i(\gtn_p^*))\leq\dim({\mathcal S}^*(\bs_i(M)))=\dim(\bs_i(M))=d_i=\hgt(\gtn_p)
$$ 
by \cite[Thms. 1, 2, p.2]{fg2}, which is a contradiction. This proves $\gamma_i(\bs_i(N))\subset \bs_i(M)$ and by symmetry we obtain $\gamma(\bs_i(N))=\bs_i(M)$.
\end{proof}

\begin{remarks}\label{unbound0r}
(i) Under the hypothesis of Theorem \ref{unbound0}, the restriction $\gamma|_{\betas N}:\betas N\to \betas M$ is also a homeomorphism because $\gamma$ and $\gamma^{-1}$ map closed points onto closed points. 

(ii) In the proof of Theorem \ref{unbound0} we have shown that the families $\bs_N:=\{\bs_i(N)\}_{i=1}^r$ and $\bs_M:=\{\bs_i(M)\}_{i=1}^r$ have the same cardinality and $\gamma(\bs_i(N))=\bs_i(M)$ for $i=1,\ldots,r$. Moreover, as $\gamma$ is a homeomorphism, $\gamma(\cl_{\betas N}(\bs_i(N)))=\cl_{\betas M}(\bs_i(M))$ for $i=1,\ldots,r$.\fina
\end{remarks}

\begin{lem}\label{hsa}
Let $\phi:{\mathcal S}^{\diam}(M)\to{\mathcal S}^{\diam}(N)$ be an isomorphism and $N_1\subset N$ a semialgebraic set. Suppose that $M_1:=\Spec(\phi)(N_1)$ is a semialgebraic subset of $M$. Then the map
$\Spec(\phi)|_{N_1}:N_1\to M_1$ is a semialgebraic homeomorphism.
\end{lem}
\begin{proof}
We assume that $M$ is bounded, so each linear projection $\pi_i:M\to\R,\ x\mapsto x_i$ belongs to ${\mathcal S}^{\diam}(M)$. Choose $f_i:=\phi(\pi_i)\in{\mathcal S}^{\diam}(N)$ and consider the semialgebraic map $\varphi:=(f_1|_{N_1},\ldots,f_m|_{N_1}):N_1\to\R^m$. It is enough to check the equality $\varphi=\Spec(\phi)|_{N_1}$, that is,
$\Spec(\phi)(\gtnd_p)=\phi^{-1}(\gtnd_p)=\gtmd_{\varphi(p)}$ for each point $p\in N_1$. Indeed, $f_i-f_i(p)\in\gtnd_p$ and so
$$
\pi_i-f_i(p)=\phi^{-1}(f_i)-f_i(p)=\phi^{-1}(f_i-f_i(p))\in\phi^{-1}(\gtnd_p)=\Spec(\phi)(\gtnd_p)=\gtmd_q
$$
for some point $q\in M_1\subset M$. Now, since $\pi_i-f_i(p)\in\gtmd_q$, it follows that $q_i-f_i(p)=0$ for $i=1,\ldots,m$, that is, $\varphi(p)=q$.
\end{proof}

Now the proof of Theorem \ref{unbound1} follows easily.

\begin{proof}[Proof of Theorem \em\ref{unbound1}]
The right to the left implication is clear. Conversely, let $\phi:{\mathcal S}(M)\to{\mathcal S}(N)$ be an isomorphism. This isomorphism induces a homeomorphism
$$
\Spec(\phi):\Specs(N)\to\Specs(M),\ \gtp\mapsto\phi^{-1}(\gtp).
$$
By Theorem \ref{unbound0} the restriction of $\Spec(\phi)$ to $N$ provides a homeomorphism between $N$ and $M$ that is semialgebraic by Lemma \ref{hsa}.
\end{proof}

We turn to bounded semialgebraic functions in order to prove Theorem \ref{bound1} and need a preliminary result.
\begin{thm}\label{bound0}
Let $\gamma:\Speca(N)\to\Speca(M)$ be a homeomorphism. Then the restriction map  $\gamma|_{N\setminus\eta(N)}:N\setminus\eta(N)\to M\setminus\eta(M)$ is a homeomorphism.
\end{thm}
\begin{proof}
All is reduced to prove $\gamma(N\setminus\eta(N))=M\setminus\eta(M)$. By symmetry it is enough to check $\gamma(N\setminus\eta(N))\subset M\setminus\eta(M)$. We claim:

\vspace{1mm}
\begin{substeps}{bound0}\label{predecessors} 
\em $N\setminus\eta(N)$ is the set of points of $\betaa N$ with at least two predecessors in $\Speca(N)$.
\end{substeps}\setcounter{substep}{0}

\vspace{1mm}
Assume this is true for a while. Then the analogous statement works for $M$ instead of $N$ and given a point $p\in N\setminus\eta(N)$, the ideal $\gtn_p^*\in\betaa N$ has two predecessors in $\Speca(N)$. If $\gtp_1\subsetneq\gtp_2$ in ${\mathcal S}^*(N)$, then $\gamma(\gtp_1)\subsetneq\gamma(\gtp_2)$ in ${\mathcal S}^*(M)$ as $\gamma$ is a homeomorphism. Thus, also $\gamma(\gtn_p^*)$ has at least two predecessors in $\Speca(M)$. Therefore $\gamma(\gtn_p^*)\in M\setminus\eta(M)$, as required. 

Now we proceed with \ref{bound0}.\ref{predecessors}. Fix a point $p\in N\setminus\eta(N)$. By the Curve Selection Lemma \cite[2.5.5]{bcr} there exist two semialgebraic paths $\alpha_1,\alpha_2:[0,1]\to\R^n$ such that $\alpha_i(0)=p$, $\alpha_i((0,1])\subset N$ and $\alpha_1((0,1])\cap\alpha_2((0,1])=\varnothing$.

Note that ${\mathcal S}(N)={\mathcal S}^*(N)_{{\mathcal W}}$ where ${\mathcal W}$ is the multiplicative set of those functions $f\in{\mathcal S}^*(N)$ such that $Z_N(f)=\varnothing$ because each $f\in{\mathcal S}(N)$ can be written as $f=(f/(1+|f|))/(1/(1+|f|))$. Since $\gtn_p\cap{\mathcal S}^*(N)=\gtn_p^*$, there exists a one-to-one correspondence, which preserves inclusions between the prime ideals of ${\mathcal S}(N)$ contained in $\gtn_p$ and those of ${\mathcal S}^*(N)$ contained in $\gtn_p^*$. Consider the prime $z$-ideals
$$
\gtp_{\alpha_i}:=\{f\in{\mathcal S}(N):\,\exists\,\veps>0\,|\ (f\circ\alpha_i)|_{(0,\,\veps)}=0\};
$$
see Proposition \ref{pmtalpha} and \ref{misp} below for a careful study. A straightforward computation shows that $\dgt_N(\gtp_{\alpha_i})=1$ while $\dgt_N(\gtn_p)=0$. Thus, $\gtp_{\alpha_i}$ has coheight $1$ in $\gtn_p$ by \ref{ht} and so the prime ideal $\gtp_{\alpha_i}\cap{\mathcal S}^*(N)$ has coheight $1$ in $\gtn_p^*$. Hence, this last one is a maximal ideal of ${\mathcal S}^*(N)$ with at least two predecessors. 

Conversely, let $\gtn^*\in\Speca(N)$ be a prime ideal with two predecessors, which implies by Proposition \ref{seteta}(ii) that $\gtn^*\not\in\eta(N)$ and so everything is reduced to check  $\gtn^*\in N$. Suppose by contradiction that $\gtn^*\in\Speca(N)\setminus N$. Then by \ref{freefixed} the only maximal ideal $\gtn$ of ${\mathcal S}(N)$ with $\gtn\cap{\mathcal S}^*(N)\subset\gtn^*$ satisfies $\gtn\cap{\mathcal S}^*(N)\subsetneq\gtn^*$. By \cite[5.2(i)]{fe1} the subchain of prime ideals of ${\mathcal S}^*(N)$ containing $\gtn\cap{\mathcal S}^*(N)$ is the same for any non refinable chain of prime ideals in ${\mathcal S}^*(N)$ ending at $\gtn^*$. In particular, since $\gtn\cap{\mathcal S}^*(N)\subsetneq\gtn^*$, the ideal $\gtn^*$ only contains one prime ideal of coheight $1$, which is a contradiction.
\end{proof}

\begin{remarks}\label{bound0r}
(i) Under the hypothesis of Theorem \ref{bound0}, the restriction map $\gamma|_{\betaa N}:\betaa N\to \betaa M$ is a homeomorphism because $\gamma$ maps closed points onto closed points. 

(ii) The homeomorphism $\gamma|:N\setminus\eta(N)\to M\setminus\eta(M)$ preserves local dimensions by \cite[Thm.2, p.2]{fg2} because it preserves the height of fixed maximal ideals. Moreover, $N\setminus\eta(N)$ and $M\setminus\eta(M)$ are dense subsets of $N$ and $M$, respectively. Let $\bs_N:=\{\bs_i(N)\}_{i=1}^r$ and $\bs_M:=\{\bs_j(M)\}_{j=1}^s$ be the families of bricks of $N$ and $M$. Using the same strategy as in \ref{unbound0}.\ref{babs}, one shows that $r=s$ and
\begin{equation}\label{ast3}
\gamma(\cl_{\Speca(N)}(\bs_i(N)))=\cl_{\Speca(M)}(\bs_i(M)) \quad \text {for}\quad i=1,\ldots,r.
\end{equation}
Since $\gamma$ maps closed points onto closed points, this implies
$$
\gamma(\cl_{\betaa N}(\bs_i(N)))=\cl_{\betaa M}(\bs_i(M))\quad \text {for}\quad i=1,\ldots,r.
$$
Moreover, by \eqref{ast3} and Theorem \ref{bound0}, $\gamma(\bs_i(N)\setminus\eta(N))=\bs_i(M)\setminus\eta(M)$ for $i=1,\ldots,r$.
\fina
\end{remarks}

Now the proof of Theorem \ref{bound1} follows straightforwardly.

\begin{proof}[Proof of Theorem \em\ref{bound1}]
The right to the left implication follows from Proposition \ref{seteta}(iii). Conversely, let $\phi:{\mathcal S}^*(M)\to{\mathcal S}^*(N)$ be an isomorphism, which induces a homeomorphism
$$
\Speca(\phi):\Speca(N)\to\Speca(M),\ \gtp\mapsto\phi^{-1}(\gtp).
$$
By Theorem \ref{bound0} $\Speca(\phi)|_{N\setminus \eta(N)}:N\setminus\eta(N)\to M\setminus\eta(M)$ is a homeomorphism, which is semialgebraic by Lemma \ref{hsa}.
\end{proof}

\begin{example}
By Theorem \ref{bound1} the rings ${\mathcal S}^*([0,1])$, ${\mathcal S}^*((0,1])$ and ${\mathcal S}^*((0,1))$ are isomorphic.
\end{example}

Next, we study $N$ and $M$ if $\Specs(N)$ and $\Speca(M)$ are homeomorphic.
\begin{thm}\label{boundunbound0}
Assume that $\Specs(N)$ and $\Speca(M)$ are homeomorphic. Then $N$ is compact and $N\setminus\eta(N)$ and $M\setminus\eta(M)$ are homeomorphic. In particular, if $N=M$, then $\Specs(N)$ and $\Speca(N)$ are homeomorphic if and only if $N$ is compact, that is, if and only if the rings ${\mathcal S}(N)$ and ${\mathcal S}^*(N)$ coincide.
\end{thm}
\begin{proof}
Suppose by contradiction that $N$ is not compact and assume that $N$ is bounded. Let $T$ be the set of isolated points of $N$, which is a finite set, and let $g\in{\mathcal S}^*(\R^n)$ be such that $Z_{\R^n}(g)=T$. By \cite[7.1]{fe1} there exists a free maximal ideal $\gtn$ of ${\mathcal S}(N)$ such that $\hgt(\gtn)=0$ and $g\not\in\gtn$. Since $N$ is dense in $\Specs(N)$, the isolated points of $\Specs(N)$ are those of $N$. Thus, $\gtn$ is a non-isolated point of $\Specs(N)$ because $g\not\in\gtn$. Let $\gamma:\Specs(N)\to\Speca(M)$ be a homeomorphism. Then $\gtm:=\gamma(\gtn)$ is a free maximal ideal of $\Speca(M)$, $\hgt(\gtm)=0$ and $\gtm$ is a non-isolated point of $\Speca(M)$. 

Let $\bs_M:=\{\bs_i(M)\}_{i=1}^r$ be the family of bricks of $M$ and define 
$$
J:=\begin{cases}
\{1,\ldots,r\}&\text{if $\dim(\bs_r(M))\geq1$,}\\ 
\{1,\ldots,r-1\}&\text{if $\dim(\bs_{r}(M))=0$.} 
\end{cases}
$$
Notice that $\gtm\in\bigcup_{i\in J}\cl_{\Speca(M)}(\bs_i(M))\cong\Speca(\bigcup_{i\in J}\bs_i(M))$ because 
$$
\Speca(M)=\bigcup_{i=1}^r\cl_{\Speca(M)}(\bs_i(M)). 
$$
Since the bricks $\bs_i(M)$ with $i\in J$ do not have isolated points, it follows from \cite[7.2]{fe1} that $\hgt(\gtm)\geq 1$, which is a contradiction. Thus, $N$ must be compact and therefore ${\mathcal S}(N)={\mathcal S}^*(N)$. Now $N\setminus\eta(N)$ and $M\setminus\eta(M)$ are by Theorem \ref{bound0} homeomorphic.

Finally, if $M=N$, then $N$ is compact and so ${\mathcal S}(N)={\mathcal S}^*(N)$. The converse is trivial.
\end{proof}

We are ready to prove Theorem \ref{boundunbound1}.

\begin{proof}[Proof of Theorem \em\ref{boundunbound1}]
Assume first that the rings ${\mathcal S}(N)$ and ${\mathcal S}^*(M)$ are isomorphic. The compactness of $N$ follows from Theorem \ref{boundunbound0}; hence, ${\mathcal S}(N)={\mathcal S}^*(N)$. Now, by Theorem \ref{bound1} the sets $N\setminus\eta(N)$ and $M\setminus\eta(M)$ are semialgebraically homeomorphic.

Conversely, by Theorem \ref{bound1} the rings ${\mathcal S}^*(N)$ and ${\mathcal S}^*(M)$ are isomorphic. Since $N$ is compact, ${\mathcal S}^*(N)={\mathcal S}(N)$ and we are done.
\end{proof}

\section{Points of the remainder associated with formal paths}\label{s4}

In this section we analyze some particular points of the remainder $\partial M:=\beta_s^*M\setminus M$ associated with formal paths. Surprisingly these points admit a countable basis of neighborhoods. In particular, these ones corresponding to semialgebraic paths play a crucial role.  For simplicity we assume in this section that $M$ is bounded.

\subsection{Extension of coefficients}\label{eoc} 
Let $F$ be a real closed field containing $\R$. There exists a (unique) semialgebraic subset $M_{F}\subset F^m$ called \em extension of $M$ to $F$ \em that satisfies $M=M_{F}\cap R^m$. The extension of semialgebraic sets depicts the natural expected behavior with respect to boolean operations, interiors, closures, boundedness, semialgebraically connected components, Transfer Principle, etc. \cite[\S5.1-3]{bcr}. Moreover, given another semialgebraic set $N\subset\R^n$ and a semialgebraic map $f:M\to N$, there exists a unique semialgebraic map $f_{F}:M_{F}\to N_{F}$ called \em extension of $f$ to $F$ \em that fulfills $f_{F}|_M=f$. The extension of semialgebraic maps enjoys the natural expected behavior with respect to direct and inverse image, continuity, injectivity, surjectivity, bijectivity, etc. \cite[\S5.1-3]{bcr}. Summarizing: `Every property that can be expressed in the first-order language of ordered fields with parameters in $\R$ can be transferred to $F$' (\cite[5.2.3]{bcr}). We refer the reader to \cite{dk0} and \cite[\S5]{bcr} for a complete study of the extension (of coefficients) to $F$. By \cite[7.3.1]{bcr} the extension of semialgebraic functions to $F$ induces a well-defined $\R$-monomorphism 
$$
{\tt i}_{M,F}:{\mathcal S}(M)\hookrightarrow{\mathcal S}(M_{F}),\ f\mapsto f_{F}. 
$$
Composing it with the evaluation homomorphism 
$$
{\rm ev}_{M_F,\tt p}:{\mathcal S}(M_F)\to F,\ g\mapsto g({\tt p})
$$ 
for ${\tt p}\in M_{F}$, we get the natural $\R$-homomorphism 
$$
\psi_{\tt p}:={\rm ev}_{M_F,\tt p}\circ {\tt i}_{M,F}:{\mathcal S}(M)\to F,\ f\mapsto f_F({\tt p}).
$$ 
Denote the restriction of the linear projection onto the $i$th coordinate to $M$ with $\pi_i:M\to\R$. In \cite[Intr. Lem. 1, p.3]{fe3} we prove that if ${\tt p}:=({\tt p}_1,\ldots,{\tt p}_m)\in M_F$, the $\R$-homomorphism $\psi_{\tt p}$ is the unique one satisfying $\pi_i\mapsto {\tt p}_i$ for $i=1,\ldots,m$. 

\subsection{Picture of the involved rings and fields}\label{picture} 

The following diagram summarizes the picture of rings and fields we use to define the free maximal ideals associated with formal and semialgebraic paths:
\begin{equation}\label{bullet2}
\xymatrix{
\R[\t]\ar@{^(->}[r]\ar@{^(->}[d]&\R[[\t]]_{\rm alg}\ar@{^(->}[r]\ar@{^(->}[d]&\R[[\t]]\ar@{^(->}[d]\\
\R(\t)\ar@{^(->}[r]&\R((\t))_{\rm alg}\ar@{^(->}[d]\ar@{^(->}[r]&\R((\t))\ar@{^(->}[d]\\
&F_0:=\R((\t^*))_{\rm alg}\ar@{^(->}[r]&F_1:=\R((\t^*))
}
\end{equation}
As usual $\R[[\t]]$ stands for the ring of formal power series in one variable with coefficients in $\R$ and $\R((\t))$ for its field of fractions. We say that a formal power series is \em algebraic \em if it is algebraic over the field of rational functions $\R(\t):=\qf(\R[\t])$. The subring (resp. subfield) of $\R[[\t]]$ (resp. $\R((\t))$) of all algebraic series is denoted with $\R[[\t]]_{\rm alg}$ (resp. $\R((\t))_{\rm alg}$). Given a formal power series $a\in\R((\t))$, we denote its order with $\omega(a)$ and the $k$-th power of the maximal ideal $(\t)$ of $\R[[\t]]$ with $(\t)^k$. We endow the rings in the diagram above with their respective unique orderings $\leq$, in which ${\tt t}>0$ (and infinitesimal).

In the following we denote the field of Puiseux series with $F_1:=\R((\t^*))$, which is the real closure of $(\R((\t)),\leq)$, and  the field of algebraic Puiseux series with $F_0:=\R((\t^*))_{\rm alg}$, which is the real closure of $(\R((\t))_{\rm alg},\leq)$. 

A \em formal path \em is a tuple $\alpha:=(\alpha_1,\ldots,\alpha_m)\in\R[[\t]]^m$. If $\alpha\in\R[[\t]]_{\rm alg}^m$, then there exists $\veps>0$ such that the map $[0,\veps]\to\R^m,\ t\mapsto\alpha(t)$ is semialgebraic. Conversely, each semialgebraic map $\alpha:[0,1]\to\R^m$ defines an element $\alpha\in\R[[\t]]_{\rm alg}^m$. Thus, the elements of $\R[[\t]]_{\rm alg}^m$ will be called \em semialgebraic paths\em. It is well-known that the positivity of a finite family of polynomials on a formal path $\alpha$ depends only on finitely many terms of the components $\alpha_j$'s of $\alpha$. More precisely,

\begin{lem}\label{positivity}
Let $\alpha\in\R[[\t]]^m$ be a formal path and $P_1,\ldots,P_r\in\R[\x]$ be such that each substitution $P_i(\alpha(\t))>0$. Then there exists a positive integer $k$ such that each substitution $P_i(\gamma(\t))>0$ for every $\gamma\in\R[[\t]]^m$ that satisfies $\|\gamma(\t)-\alpha(\t)\|^2\in(\t)^{2k}$.
\end{lem}
\begin{proof}
Since $P_i(\alpha(\t))>0$, we can write $P_i(\alpha(\t))=a_i\t^{q_i}+\cdots$ where $q_i=\omega(P_i(\alpha(\t)))$ and $a_i>0$. Next, write $H_i(\x,\y,\s)=P_i(\x+\s\y)$ and let $F_{ij}\in\R[\x,\y]$ be polynomials such that $H_i(\x,\y,\s)=P_i(\x)+\sum_{j=1}^{s_i}\s^j F_{ij}(\x,\y)$. Define $k:=1+\max\{q_i: 1\leq i\leq r\}$. If $\gamma\in\R[[\t]]^m$ satisfies $\|\gamma(\t)-\alpha(\t)\|^2\in(\t)^{2k}$, then there exists $\beta:=(\beta_1,\ldots,\beta_m)\in\R[[\t]]^m$ such that $\gamma(\t)=\alpha(\t)+\t^{k}\beta(\t)$. After the substitution $\x=\alpha(\t)$, $\s=\t^{k}$ and $\z=\beta(\t)$, we have
\begin{equation*}
\begin{split}
P_i(\gamma(\t))&=H_i(\alpha(\t),\beta(\t),\t^{k})=P_i(\alpha(\t))+\sum_{j=1}^{s_i}(\t)^{kj} F_{ij}(\alpha(\t),\beta(\t))=a_i\t^{q_i}+\cdots;
\end{split}
\end{equation*}
hence, $P_i(\gamma(\t))>0$ for $i=1,\ldots, r$.
\end{proof}

\begin{cor}\label{compact}
Let $\alpha\in M_{F_1}$ be a formal path. Then 
\begin{itemize}
\item[(i)] The point $\alpha(0)\in\cl_{\R^m}(M)$.
\item[(ii)] If $p:=\alpha(0)\in M$, there exists a compact semialgebraic set $K\subset M$ such that $\alpha\in K_{F_1}$ and $p\in K$.
\end{itemize}
\end{cor}
\begin{proof}
Choose polynomials $f_1,\ldots,f_r,g\in\R[\x]$ such that $M_1:=\{f_1>0,\ldots,f_r>0,g=0\}\subset M$ satisfies $\alpha\in M_{1,F}$. 

(i) By Lemma \ref{positivity} and Artin's approximation theorem for Nash series there exists a semialgebraic path $\gamma\in M_{1,F_1}$ and $\alpha(0)=\gamma(0)$; hence, for $\veps>0$ small enough we have $\im(\gamma|_{(0,\,\veps)})\subset M_1$ and $\alpha(0)=\gamma(0)\in\gamma([0,\veps])=\cl_{\R^m}(\im\gamma|_{(0,\,\veps)})\subset\cl_{\R^m}(M_1)\subset\cl_{\R^m}(M)$.

(ii) Let $k\geq1$ be an integer such that every $\gamma\in\R[[\t]]^m$ with $\|\gamma(\t)-\alpha(\t)\|^2\in(\t)^{2k}$ satisfies $f_i(\gamma(\t))>0$ for $i=1,\ldots,r$ (see Lemma \ref{positivity}). Choose $\mu:=(\mu_1,\ldots,\mu_m)\in\R[\t]^m$ such that $\deg(\mu_i)\leq k$ and $\|\mu(\t)-\alpha(\t)\|^2\in(\t)^{2(k+1)}$. By the choice of $k$, for all $s\in\R^m$ the polynomial path $\gamma_s(\t)=\mu(\t)+\t^{k+1}s\in\R[\t]^m$ satisfies $\gamma_s\in\{f_1>0,\ldots,f_r>0\}_{F_1}$.

We write $f_i(\gamma_{\s}(\t))=a_i\t^{p_i}+\t^{p_i+1}h_i(\t,\s)$ where $h_i\in\R[\t,\s]$ for each $i=1,\ldots,r$ and define
$$
s_0:=\frac{\alpha(\t)-\mu(\t)}{\t^{k+1}}\Big|_{{\tt t}=0}\in\R^m.
$$
Let $L>0$ be such that each $|h_i(t,s)|<L$ for all $(t,s)\in[0,1]\times\ol\Bb(s_0,1)$ and let $0<\veps<1$ be a real number such that each $a_it^{p_i}-t^{p_i+1}L>0$ for all $t\in(0,\veps]$. Notice that $f_i(\gamma_s(t))>0$ for all $(t,s)\in(0,\veps]\times\ol\Bb(s_0,1)$. Now consider the semialgebraic map 
$$
h:[0,\veps]\times\ol\Bb(s_0,1)\subset\R^{m+1}\to\R^m,\ (t,s)\mapsto\gamma_s(t)=\mu(t)+t^{k+1}s
$$
and the compact semialgebraic set $K_0:=h([0,\veps]\times\ol\Bb(s_0,1))$. Since $h(\{0\}\times\ol\Bb(s_0,1))=\{p\}$, we deduce $\{p\}\subset K_0\subset\{f_1>0,\ldots,f_r>0\}\cup\{p\}$. Thus, 
$$
K:=K_0\cap\{g=0\}\subset M_1\cup\{p\}\subset M
$$ 
is a compact semialgebraic set that contains $p$ because by (i) $p\in\cl_{\R^m}(M_1)$ and so $g(p)=0$. Finally, we must check $\alpha\in K_{F_1}$. Indeed, $\alpha(\t)=\mu(\t)+\t^{k+1}\zeta(\t)$ where $\zeta(\t)\in\R[[\t]]^m$ and $\zeta(0)=s_0$; hence, $\zeta\in\Bb(s_0,1)_{F_1}$ and so $\alpha\in K_{0,F_1}$. Since $\alpha\in\{g=0\}_{F_1}$, we conclude $\alpha\in K_{0,F_1}\cap\{g=0\}_{F_1}=K_{F_1}$, as wanted.
\end{proof}

\subsection{Free maximal ideals associated with formal and semialgebraic paths}

Let $\alpha\in M_{F_1}$ be a formal path. By \ref{eoc} there exists a unique homomorphism $\psi_\alpha:{\mathcal S}(M)\to F_1$ such that $\psi_\alpha(\pi_i)=\alpha_i$. We claim: $\psi_\alpha({\mathcal S}^*(M))\subset\R[[{\tt t}^*]]$. 

Indeed, let $f\in{\mathcal S}^*(M)$ and $L>0$ be a constant such that $|f|<L$. Then $L-f>0$ and $f+L>0$; pick $h_1,h_2\in{\mathcal S}^*(M)$ such that $h_1^2=L-f$ and $h_2^2=f+L$. Then
$$
L-\psi_\alpha(f)=\psi_\alpha(h_1^2)=\psi_\alpha(h_1)^2\geq0\quad\text{and}\quad\psi_\alpha(f)+L=\psi_\alpha(h_2^2)=\psi_\alpha(h_2)^2\geq0,
$$
so $|\psi_\alpha(f)|\leq L$; hence, $\psi_\alpha(f)\in\R[[\t^*]]$.

\subsubsection{Free maximal ideals associated with formal paths}\label{mifp} 
Consider the `evaluation' 
$$
{\rm ev}_0:\R[[\t^*]]\to\R,\ \zeta\mapsto\zeta(0)
$$ 
and the $\R$-epimorphism
$$
\varphi_\alpha:={\rm ev}_0\circ\psi_\alpha:{\mathcal S}^*(M)\to\R,\ f\mapsto({\rm ev}_0\circ\psi_\alpha)(f)=\psi_\alpha(f)(0). 
$$
Then $\gtm_\alpha^*:=\ker(\varphi_\alpha)$ is a maximal ideal of ${\mathcal S}^*(M)$. As one can expect, $\gtm_{\alpha}^*=\gtm_{\alpha(0)}^*$ if $\alpha(0)\in M$ and $\gtm^*_\alpha$ is a free maximal ideal of ${\mathcal S}^*(M)$ if $\alpha(0)\in\cl_{\R^m}(M)\setminus M$. In the latter case we call $\gtm_\alpha^*$ \em the free maximal ideal of ${\mathcal S}^*(M)$ associated with $\alpha$\em. We denote the collection of all free maximal ideals of ${\mathcal S}^*(M)$ associated with formal paths with $\widehat{\partial}M\subset\partial M$.

\subsubsection{}
Let us find the maximal ideal $\gtm_\alpha$ of ${\mathcal S}(M)$ corresponding to $\gtm_\alpha^*$ via the homeomorphism $\Phi$ introduced in \ref{homeo}. We call $\gtm_\alpha$ \em the free maximal ideal of ${\mathcal S}(M)$ associated with $\alpha$\em. 

\begin{prop}\label{pmtalpha}
Let $\alpha\in M_{F_1}$ be a formal path and consider $\gtp_\alpha:=\ker(\psi_\alpha)$. Then 
\begin{itemize}
\item[(i)] $\gtp_\alpha=\gtm_\alpha$ is the free maximal ideal of ${\mathcal S}(M)$ satisfying $\gtm_{\alpha}\cap{\mathcal S}^*(M)\subset\gtm_{\alpha}^*$ if $\alpha(0)\not\in M$. 
\item[(ii)]$\gtp_\alpha$ is a prime $z$-ideal of ${\mathcal S}(M)$ of coheight $1$ contained in $\gtm_{\alpha(0)}$ if $\alpha(0)\in M$. 
\end{itemize}
\end{prop}
\begin{proof}
It is straightforward to check that $\gtp_\alpha$ is a prime $z$-ideal and $\gtp_\alpha\cap{\mathcal S}^*(M)\subset\ker(\varphi_\alpha)=\gtm^*_\alpha$. Let us show next that $\gtp_\alpha$ is a maximal ideal. Otherwise let $\gtq$ be a prime ideal of ${\mathcal S}(M)$ such that $\gtp_\alpha\subsetneq\gtq$ and choose $f\in\gtq\setminus\gtp_\alpha$. Taking $f/(1+f^2)$ instead of $f$, we may assume that $f$ is bounded on $M$.

As $\psi_{\alpha}(f)\in\R[[\t^*]]\setminus\{0\}$, we write $\psi_{\alpha}(f)(\t):=a\t^b+\cdots$ where $a\neq 0$, $b:=\omega(\psi_{\alpha}(f))\in\Q^+$ and $\|\alpha(\t)-p\|:=c\t^d+\cdots$ where $p:=\alpha(0)$, $c\neq 0$ and $d:=\omega(\|\alpha(\t)-p\|)\in\Q^+$. Consider
$$
Z:=\Big\{x\in M:\,\frac{|a|}{2c^{b/d}}\|x-p\|^{b/d}\leq |f(x)|\Big\}
$$ 
and pick $g:=\dist(\cdot,Z)\in{\mathcal S}(M)$, which satisfies $Z_M(g)=Z$. Observe 
$$
\frac{|a|}{2c^{b/d}}\|\alpha(\t)-p\|^{b/d}=\frac{|a|}{2}\t^b+\cdots\qquad\text{and}\qquad\psi_{\alpha}(|f|)(\t)=|a|\t^b+\cdots,
$$
so $\alpha\in Z_{F_1}$; hence, $\psi_\alpha(g)=g_{F_1}(\alpha)=0$ or equivalently $g\in\gtp_\alpha\subset\gtq$.

The zero set of $h:=f^2+g^2\in\gtq$ satisfies
$$
Z_M(h)=\left\{\begin{array}{ll}
\!p&\text{ if $p\in M$,}\\[4pt]
\!\varnothing&\text{ if $p\in\cl_{\R^m}(M)\setminus M$.}
\end{array}\right.
$$ 

(i) In particular, if $p\in\cl_{\R^m}(M)\setminus M$, the function $h\in\gtq$ is a unit in ${\mathcal S}(M)$, which is a contradiction. Thus, $\gtp_\alpha$ is a maximal ideal of ${\mathcal S}(M)$ satisfying $\gtp_\alpha\cap{\mathcal S}^*(M)\subset\gtm^*_\alpha$. 

(ii) On the other hand, if $p\in M$, there exists by Corollary \ref{compact}(ii) a compact semialgebraic set $K\subset M$ such that $p\in K$ and $\alpha\in K_{F_1}$. Since $K\subset M$ is closed in $M$, the homomorphism $\phi:{\mathcal S}(M)\to{\mathcal S}(K),\ f\mapsto f\circ{\tt j}$ induced by the inclusion ${\tt j}:K\hookrightarrow M$ is surjective \cite{dk}. Notice that $\ker\phi\subset\gtp_\alpha\subset\gtq$; hence, $\gtq=\phi^{-1}(\phi(\gtq))$. Moreover, since $K$ is compact, $\phi(\gtq)$ is a $z$-ideal and by \cite[4.1]{fg3}) $\gtq=\phi^{-1}(\phi(\gtq))$ is also a prime $z$-ideal. Since $h\in\gtq$ and $Z_M(h)=\{p\}$, we conclude $\gtq=\gtm_p$ and so $\gtp_\alpha$ is a prime $z$-ideal of coheight one.
\end{proof}

\begin{cor}\label{ppp}
Let $\alpha\in M_{F_1}$ be a formal path such that $\alpha(0)\in\cl_{\R^m}(M)\setminus M$. Then there does not exist any prime ideal between $\gtm_\alpha\cap{\mathcal S}^*(M)$ and $\gtm_\alpha^*$.
\end{cor}
\begin{proof}
Let $\gtp_0:=\gtm_\alpha\cap{\mathcal S}^*(M)\subsetneq\cdots\subsetneq\gtp_r=\gtm_\alpha^*$ be the collection of all prime ideals of ${\mathcal S}^*(M)$ containing $\gtp_0$. By \cite[1.A.2]{fg2} there exists a semialgebraic compactification $(X,{\tt j})$ of $M$ and a chain of prime ideals $\gtq_0\subsetneq\cdots\subsetneq\gtq_r$ of ${\mathcal S}(X)$ such that $\gtq_i=\gtp_i\cap{\mathcal S}(X)$. Assume $X\subset\R^m$ and notice $\gtq_0=\ker(\psi_{{\tt j}\circ\alpha})$ where ${\tt j}\circ\alpha\in\R[[\t^*]]^m$. After reparametrizing $\alpha$ if necessary, we may assume ${\tt j}\circ\alpha\in\R[[\t]]^m$. By Proposition \ref{pmtalpha} $\gtq_0$ has coheight $1$, that is, $r=1$, as wanted.
\end{proof}

\subsubsection{Free maximal ideals associated with semialgebraic paths}\label{misp}
The collection of all free maximal ideals $\gtm_\alpha$ of ${\mathcal S}^*(M)$ corresponding to semialgebraic paths $\alpha\in M_{F_0}$ is denoted with $\widetilde{\partial} M$. Of course, $\widetilde{\partial} M\subset\widehat{\partial}M\subset\partial M$ and in general both inclusions are strict and the differences are `large' (see the proof in Appendix \ref{diferencias}). The uniqueness of the homomorphism $\psi_\alpha$ guarantees that if $\alpha\in M_{F_0}$ is a semialgebraic path, the $\R$-homomorphism $\psi_\alpha:{\mathcal S}(M)\to F_0$ is defined by $f\mapsto f\circ\alpha$. Moreover, if $\alpha\in M_{F_0}$ and $\alpha(0)\in\cl_{\R^m}(M)\setminus M$, then
\begin{align*}
\gtm_\alpha^*&=\{f\in{\mathcal S}^*(M):\,\lim_{t\to0^+}(f\circ\alpha)(t)=0\}\\
\gtm_\alpha&=\{f\in{\mathcal S}(M):\,\exists\,\veps>0\text{ such that }(f\circ\alpha)|_{(0,\,\veps]}=0\}.
\end{align*}

Let us show next that the prime $z$-ideals of semialgebraic depth equal to $1$ in ${\mathcal S}(M)$ are the prime ideals $\gtp_\alpha:=\ker(\psi_\alpha)$ where $\alpha\in M_{F_0}$ is a semialgebraic path. Of course, if $\alpha\in M_{F_0}$ is a non-trivial semialgebraic path, we may assume $\alpha_1(\t)=\t^p$ and $\alpha_i\in\R[[\t]]_{\rm alg}$. Thus, $\qf({\mathcal S}^*(M)/\gtp_\alpha)=\R((\t^*))_{\rm alg}$ is the real closure of $\R[\t]$ and by \cite[Thm.3, p.3]{fg2} we conclude ${\tt d}_M(\gtp_\alpha)=1$. The converse is the following lemma.

\begin{lem}
Let $\gtp$ be a prime $z$-ideal of ${\mathcal S}(M)$ with semialgebraic depth ${\tt d}_M(\gtp)=1$. Then there exists a semialgebraic path $\alpha\in M_{F_0}$ such that $\gtp=\gtp_\alpha$.
\end{lem}
\begin{proof}
Choose $f\in\gtp$ such that $\dim(Z_M(f))=1$ and $Z_M(f)$ has no isolated points. By \cite[2.9.10]{bcr} $Z_M(f)$ is the disjoint union of finitely many points $p_{k}$ and a finite number of Nash curves $M_{k}$ where each of them is Nash diffeomorphic to an open interval $(0,1)$. In fact, $Z_M(f)=\bigcup_k\cl_{M}(M_k)$ and we may assume that each $\cl_{M}(M_k)$ is semialgebraically homeomorphic either to $(0,1]$ or $[0,1]$. Write $\cl_M(M_k)=Z_M(g_k)$ for some $g_k\in{\mathcal S}(M)$, so $Z_M(f)=Z_M(g)$ where $g:=g_1\cdots g_k$. Since $\gtp$ is a prime $z$-ideal, we obtain $g\in\gtp$ and may assume $g_1\in\gtp$. Denote $N:=\cl_M(M_1)=Z_M(g_1)$ and recall that the homomorphism $\theta:{\mathcal S}(M)\to{\mathcal S}(N)$ is surjective by \cite{dk}. As $\ker(\theta)\subset\gtp$, there exists a prime ideal $\gtq$ of ${\mathcal S}(N)$ such that
${\mathcal S}(M)/\gtp\cong{\mathcal S}(N)/\gtq$. In fact, $\gtq$ is a minimal prime ideal of ${\mathcal S}(N)$ by \cite[4.1]{fe1}. 

By construction, $N$ is semialgebraically homeomorphic to either $I:=(0,1]$ or $I:=[0,1]$ via a semialgebraic homeomorphism $h:I\to N$; hence, the rings ${\mathcal S}(I)$ and ${\mathcal S}(N)$ are isomorphic. We may assume by \cite[4.1]{fe1} that $\gtq$ corresponds to the minimal prime ideal $\gtp_0:=\{a\in{\mathcal S}(I):\ \exists\,\veps>0,\ a|_{(0,\veps]}=0\}$. Moreover, as $N$ is bounded, $h$ can be extended to a semialgebraic path $\alpha:[0,1]\to \R^m$ such that $\alpha((0,1])\subset N\subset M$; hence, $\alpha\in\R[[\t]]_{\rm alg}^m\cap M_{F_0}$. Now it is straightforward to check $\gtp=\gtp_\alpha$, as required.
\end{proof}

\subsection{Density of $\widetilde{\partial}M$ in $\partial M$.}\label{ppoints}
Although $\widetilde{\partial} M\neq\partial M$, we prove easily that the latter is a dense subset of $\partial M$. 

\begin{lem}\label{deltam}
\em (i) \em Let $f_i\in{\mathcal S}^*(M)$ and $\widehat{f_i}:\betaa M\to \R$ be the unique continuous extension of $f_i$ to $\betaa M$. Then $(\widehat{f_1},\ldots,\widehat{f_r})(\widetilde{\partial}M)=(\widehat{f_1},\ldots,\widehat{f_r})(\partial M)$ for each $r\geq1$.

\em (ii) \em A function $f\in{\mathcal S}^*(M)$ is a unit if and only if $Z_M(f)=\varnothing$ and $f\not\in\gtm_\alpha^*$ for all $\gtm_\alpha^*\in\widetilde{\partial}M$. 

\em (iii) \em The set $\widetilde{\partial}M$ is dense in $\partial M$.
\end{lem}
\begin{proof} 
(i) As $M$ is bounded, $\cl_{\R^m}(M)$ is a semialgebraic compactification of $M$. Thus, there exists by \cite[4.6]{fg5} a surjective continuous map $\rho:\betaa M\to\cl_{\R^m}(M)$ that is the identity on $M$. Fix $\gtm^*\in\partial M$ and observe that by \cite[4.3(i)]{fg5} $p:=\rho(\gtm^*)\in\cl_{\R^m}(M)\setminus M$. Consider the proper map $\Psi:=(\rho,\widehat{f}):\betaa M\to\R^{m+r}$ where we abbreviate $f:=(f_1,\ldots,f_r)$ and $\widehat{f}:=(\widehat{f_1},\ldots,\widehat{f_r})$ (see \ref{cocr} for the definition of $\widehat{f_i}$). Clearly, $\Psi(M)$ is the graph $\Gamma$ of $f$ and since $\Psi$ is proper, 
$$
\im\Psi=\Psi(\cl_{\betaa M}(M))=\cl_{\R^{m+r}}(\Psi(M))=\cl_{\R^{m+r}}(\Gamma).
$$
Again by \cite[4.3(i)]{fg5} $q:=\Psi(\gtm^*)=(\rho(\gtm^*),\widehat{f}(\gtm^*))=(p,a)\in\cl_{\R^{n+r}}(\Gamma)\setminus\Gamma\subset\R^m\times\R^r$.

By the Curve Selection Lemma \cite[2.5.5]{bcr}, there are semialgebraic paths $\alpha:[0,1]\to\R^m$ and $\mu:[0,1]\to\R^r$ such that $\alpha((0,1])\subset M$, $\mu|_{(0,\,1]}=(f\circ\alpha)|_{(0,\,1]}$ and $(\alpha(0),\mu(0))=q$. Hence,
$$
\widehat{f}(\gtm^*)=a=\mu(0)=\lim_{t\to0^+}\mu(t)=\lim_{t\to0^+}(f\circ\alpha)(t)=\widehat{f}(\gtm_\alpha^*)
$$
where $\gtm_\alpha^*\in \widetilde{\partial}M$ because $\lim_{t\to0^+}\alpha(t)=p\not\in M$. 

(ii) Observe that $f\in{\mathcal S}^*(M)$ is a unit if and only if $0\notin\widehat{f}(\betaa M)=f(M)\cup\widehat{f}(\partial M)=f(M)\cup\widehat{f}(\widetilde{\partial}M)$ (see assertion (i)), which proves the statement.

(iii) We have to check that for every $f\in {\mathcal S}^*(M)$ such that $\di_{\betaa M}(f)\not \subset M$ the intersection $\di_{\betaa M}(f)\cap\widetilde{\partial}M$ is non empty. Otherwise, $\widetilde{\partial}M\subset\ceros_{\betaa M}(f)$ and by part (i) we obtain $\{0\}=\widehat{f}(\widetilde{\partial}M)=\widehat{f}(\partial M)$ or equivalently $\di_{\betaa M}(f)\subset M$, which is a contradiction.
\end{proof}

A nice consequence of the previous result is the following characterization of the properness of a surjective semialgebraic map. Similar results can be found in \cite[2.1, 2.2]{b}.

\begin{cor}\label{propern}
A surjective semialgebraic map $g:N\to M$ is proper if and only if $\betaa\,g(\partial N)=\partial M$.
\end{cor}
\begin{proof}
Note first that $\betaa\,g:\betaa N\to\betaa M$ is a proper map since it is continuous and both $\betaa N$ and $\betaa M$ are compact and Hausdorff spaces; in particular, since $g$ is surjective and $M$ is dense in $\betaa M$, also $\betaa\,g$ is surjective. We assume again that $M$ is bounded. 

If $\betaa\,g(\partial N)=\partial M$, we deduce that $g=\betaa\,g|_{N}:N=(\betaa\,g)^{-1}(M)\to M$ is also proper. Conversely, assume that $g$ is proper. Since $g$ and $\betaa\,g$ are surjective, we obtain $\partial M\subset\betaa\,g(\partial N)$. To prove the other inclusion, we proceed as follows.

Denote $\widehat{g}:=(\widehat{g_1},\ldots,\widehat{g_m}):\betaa N\to\R^m$ where $\widehat{g_i}$ is the (unique continuous) extension of the component $g_i$ of $g$ to $\betaa N$. Suppose by contradiction that there exists a point $\gtn^*\in\partial N$ such that $p:=\betaa\,g(\gtn^*)\in M$. By Lemma \ref{deltam}(i) there exists $\gtn^*_\alpha\in\widetilde{\partial}N$ such that $\widehat{g}(\gtn^*_{\alpha})=p$; hence, $\betaa\,g(\gtn^*_\alpha)=\lim_{t\to0^+}(g\circ\alpha)(t)=\widehat{g}(\gtn^*_{\alpha})=p\in M$. Since $g$ is proper, $\alpha(0)=\lim_{t\to0^+}\alpha(t)$ belongs to $N$, which contradicts the fact that $\gtn^*_\alpha\in\widetilde{\partial}N$.
\end{proof}

\subsection{Points of the maximal spectrum with a countable basis of neighborhoods}\label{furapp}

A reasonable strategy to decide if $\betaa M$ determines the topological type of $M\setminus\eta(M)$ consists of searching a topological condition to distinguish the points of $M\setminus\eta(M)$ among those of $\betaa M$. In \cite[9.6-7]{gj} the authors prove that if $X$ is a topological space whose points are $G_{\delta}$, then $X$, viewed as a subset of its Stone--\v{C}ech compactification $\beta X$, is the set of $G_{\delta}$--points of $\beta X$. We show next that all points of a semialgebraic set $M$ have a countable basis of neighborhoods in $\betaa M$. Of course, this is trivially true for the points of $M_{\lc}$ as $M_{\lc}$ is open in $\betaa M$. However, the points of $\rho_1(M)$ require a more careful analysis.

\begin{lem}\label{fc}
Let $f\in{\mathcal S^*(M)}$ and $\widehat{f}:\betaa M\to\R$ be its unique continuous extension to $\betaa M$. Let $\gtm\in\betaa M$ be such that $c:=\widehat{f}(\gtm)>0$. Then $\cl_{\betaa M}(f^{-1}((\tfrac{c}{2},+\infty)))=\cl_{\betaa M}(\widehat{f}^{-1}(\tfrac{c}{2},+\infty))$ is a closed neighborhood of $\gtm$ in $\betaa M$.
\end{lem}
\begin{proof}
It is enough to check that $\widehat{f}^{-1}((\tfrac{c}{2},+\infty))\subset\cl_{\betaa M}(f^{-1}(\tfrac{c}{2},+\infty))$. Fix a point $\gtn\in\widehat{f}^{-1}((\tfrac{c}{2},+\infty))$ and let $V$ be a neighborhood of $\gtn$ in $\betaa M$. Then $V\cap\widehat{f}^{-1}((\tfrac{c}{2},+\infty))$ is also a neighborhood of $\gtn$ in $\betaa M$. Since $M$ is dense in $\betaa M$, the intersection
$$
V\cap\widehat{f}^{-1}((\tfrac{c}{2},+\infty))\cap M=V\cap f^{-1}((\tfrac{c}{2},+\infty))\neq\varnothing.
$$
Thus, $\gtn\in\cl_{\betaa M}(f^{-1}((\tfrac{c}{2},+\infty)))$, as wanted.
\end{proof}

\begin{prop}\label{cn}
Let $p\in M$ and $\{U_k\}_k$ be a countable basis of neighborhoods of $p$ in $M$. Then $\{\cl_{\betaa M}(U_k)\}_k$ is a countable basis of neighborhoods of $p$ in $\betaa M$.
\end{prop}
\begin{proof}
Let $W$ be an open neighborhood of $p$ in $\betaa M$. Then there exists a bounded semialgebraic function $f\in{\mathcal S}^*(M)$ such that $p\in{\mathcal D}_{\betaa M}(f)\subset W$. Let $\widehat{f}:\betaa M\to\R$ be the unique continuous extension of $f$ to $\betaa M$. We may assume $\widehat{f}(p)=c>0$ and observe that $f^{-1}((\tfrac{c}{2},+\infty))$ is an open neighborhood of $p$ in $M$. Thus, there exists $k\geq 1$ such that $p\in U_k\subset f^{-1}((\tfrac{c}{2},+\infty))$. Therefore
$$
\cl_{\betaa M}(U_k)\subset\cl_{\betaa M}(f^{-1}((\tfrac{c}{2},+\infty)))\subset\cl_{\betaa M}(\widehat{f}^{-1}((\tfrac{c}{2},+\infty)))\subset\widehat{f}^{-1}([\tfrac{c}{2},+\infty))\subset W.
$$

To finish, let us see that each set $\cl_{\betaa M}(U_k)$ is a neighborhood of $p$ in $\betaa M$. Let $W_k$ be a neighborhood of $p$ in $\betaa M$ such that $U_k=W_k\cap M$. Let $g\in{\mathcal S}^*(M)$ be such that $p\in{\mathcal D}_{\betaa M}(g)\subset W_k$. Then $p\in D_M(g)\subset U_k$. We may assume $r=g(p)>0$ and so $p\in\widehat{g}^{-1}((r/2,+\infty))\subset W_k$. Thus, $p\in g^{-1}((r/2,+\infty))\subset W_k\cap M=U_k $ and by Lemma \ref{fc} 
$$
p\in\widehat{g}^{-1}((r/2,+\infty))\subset\cl_{\betaa M}(\widehat{g}^{-1}((r/2,+\infty)))=\cl_{\betaa M}(g^{-1}((r/2,+\infty)))\subset\cl_{\betaa M}(U_k).
$$
Hence, $\cl_{\betaa M}(U_k)$ is a neighborhood of $p$ in $\betaa M$, as wanted.
\end{proof}

We prove next that there also exist a lot of points in $\betaa M\setminus M$, which have a countable basis of neighborhoods in $\betaa M$; hence, the strategy for rings of continuous functions employed in \cite[9.6-7]{gj} to distinguish the points of $M$ among those of $\betaa M$ becomes useless in our context.

\begin{prop}\label{gpn3}
Each point of $\widehat{\partial}M$ has a countable basis of neighborhoods in $\betaa M$.
\end{prop}
\begin{proof}
Let $\alpha:=(\alpha_1,\ldots,\alpha_m)\in M_{F_1}$ be a formal path such that $\alpha(0)\in\cl_{\R^m}(M)\setminus M$. Our aim is to construct a countable basis of neighborhoods for $\gtm_\alpha^*$ in $\betaa M$. 

\vspace{1mm}\setcounter{substep}{0}
\begin{substeps}{gpn3}\label{piccolino1}
We may assume: \em $M\subset\{x_1>0\}$, $\alpha(0)=0$ and $\alpha_1(\t)=\t$\em. 
\end{substeps}

\vspace{1mm}
Indeed, after a change of coordinates in $\R^m$ we may assume $\alpha(0)=0$ and that $\alpha_1$ is not a constant. Considering the embedding of $\R^m$ in $\R^{m+1}$ given by 
$$
(x_1,\ldots,x_m)\mapsto(x_1^2+\cdots+x_m^2,x_1,\ldots,x_m)=(y_1,\ldots,y_{m+1}),
$$ 
we can suppose $M\subset\{y_1>0\}$. After reparametrizing $\alpha$, we assume $\alpha_1(\t)=\t^p$ for some integer $p\geq1$. This in combination with the new change of coordinates 
$$
h:(0,+\infty)\times\R^{m}\to(0,+\infty)\times\R^{m},\ (y_1,y_2,\ldots,y_{m+1})\mapsto(\sqrt[p]{y_1},y_2,\ldots,y_{m+1})
$$
allows us to suppose $\alpha_1(\t)=\t$.

\vspace{1mm}
\begin{substeps}{gpn3}\label{piccolino2}
For each integer $\ell\geq 1$ consider polynomials $\gamma_{2\ell},\ldots,\gamma_{m\ell}\in\R[\t]$ such that $\alpha_j-\gamma_{j\ell}\in(\t)^{\ell+2}\subset\R[[\t]]$ and let $L_\ell>0$ be such that $|\gamma_{j\ell}(\t)|<L_\ell$ for $|t|\leq1$ and $j=2,\ldots,m$. Denote $\gamma_\ell(\t):=(\t,\gamma_{2\ell}(\t),\ldots,\gamma_{m\ell}(\t))$ and consider the family of semialgebraic functions on $M$
$$
\begin{cases}
f_\ell(x):=x_1^{2\ell+2}-\|x-\gamma_\ell(x_1)\|^2=x_1^{2\ell+2}-\sum_{j=2}^m(x_j-\gamma_{j\ell}(x_1))^2,&\text{for}\ \ell\geq 1,\\[4pt] 
h_k(x):=\frac{1}{k^2}-x_1^2,&\text{for}\ k\geq 1
\end{cases}
$$
and the family of open subsets $U_{\ell,k}:={\mathcal D}_{\betaa M}(f_\ell+|f_\ell|,h_k+|h_k|)$ of $\betaa M$. Note that $\gtm_\alpha^*\in U_{\ell,k}$ for all $\ell,k\geq 1$. 
\end{substeps}

\vspace{1mm}
\begin{substeps}{gpn3}\label{piccolino20}
Our goal is to see: \em $\{U_{\ell,k}\}_{\ell,k}$ is a basis of neighborhoods of $\gtm_\alpha^*$ in $\betaa M$\em. 
\end{substeps}

\vspace{1mm}
Fix $g\in{\mathcal S}^*(M)$ such that $\gtm^*_\alpha\in{\mathcal D}_{\betaa M}(g)$ and assume $\widehat{g}(\gtm^*_\alpha)=c>0$. We write $V:=g^{-1}((\tfrac{c}{2},+\infty))$. Notice that $\alpha\in V_{F_1}$ and choose polynomials $g_1,\ldots,g_r\in\R[\x]$ such that $V_1:=\{g_1>0,\ldots,g_r>0\}$ satisfies $\alpha\in(V_1\cap M)_{F_1}\subset V_{F_1}$. Consider the new variables $\s$, $\y:=(\y_1,\ldots,\y_m)$ and $\z:=(\z_1,\ldots,\z_m)$, write $\x=\y+\s\z$ and 
$$
g_i(\x)=g_i(\y+\s\z)=g_i(\y)+\s H_i(\s,\y,\z)
$$
where $H_i(\s,\y,\z):=\sum_{j=1}^{s_i-1}h_{ij}(\y,\z)\s^j$ for some polynomials $h_{i1},\ldots,h_{is_i}\in\R[\y,\z]$. Let $C_\ell>0$ be a large enough real number such that 
$$
\text{$|H_i(s,y,z)|<C_\ell$\quad for $|s|\leq1,|z_j|\leq1,|y_1|\leq1,|y_2|\leq L_\ell,\ldots,|y_m|\leq L_\ell$},
$$
$j=1,\ldots,m$ and $i=1,\ldots,r$. 

\vspace{1mm}
\begin{substeps}{gpn3}\label{piccolo}
By (a slight modification of) Lemma \ref{positivity}: \em There exists $\ell_0\geq1$ such that for all $\ell\geq \ell_0$ every formal path $\eta\in\R[[\t]]^m$ with $\|\eta(\t)-\alpha(\t^p)\|^2\in(\t)^{2\ell p}$ for some $p\geq1$ satisfies $g_i(\eta(\t))>0$ for $i=1,\ldots,r$\em. In particular: \em If $\ell\geq\ell_0$, then $g_i(\gamma_{\ell}(t))>0$ for $t>0$ small enough\em. 
\end{substeps}

\vspace{1mm}
\begin{substeps}{gpn3}\label{piccolo2}
Denote $\ell:=1+\max\{\ell_0,\omega(g_i(\alpha(\t))):\ i=1,\ldots,r\}$ and choose $k_0\geq 1$ such that $g_i(\gamma_{\ell}(t))>0$ for $i=1,\ldots,r$ if $0<t<1/k_0$. Since $\ell>\omega(g_i(\alpha(\t)))$ for $\ i=1,\ldots,r$, there exists $k\geq k_0$ such that $
g_i(\gamma_{\ell}(t))-t^{\ell+1}C_\ell>0$ for $0<t\leq1/k$ and $i=1,\ldots,r$. 
\end{substeps}

\vspace{1mm}
\begin{substeps}{gpn3}\label{piccolo3}
For our purposes it is enough to check: $U_{\ell,k}\subset\widehat{g}^{-1}([\tfrac{c}{2},+\infty))$. 
\end{substeps}

\vspace{1mm}
Indeed, fix a point $x\in U_{\ell,k}\cap M$. Then $0<x_1<1/k$ and $\sum_{j=2}^m(x_j-\gamma_{j\ell}(x_1))^2 < x_1^{2\ell+2}$. Thus, $|x_j-\gamma_{j\ell}(x_1)|<x_1^{\ell+1}$ for $j=2,\ldots,m$ and so $x_j=\gamma_{j\ell}(x_1)+\rho_jx_1^{\ell+1}$ for some $\rho_j\in\R$ such that $|\rho_j|<1$. Write $\rho:=(0,\rho_2,\ldots,\rho_m)$ and observe that by \ref{gpn3}.\ref{piccolo2}
$$
g_i(x)=g_i(\gamma_{\ell}(x_1))+x_1^{\ell+1}H_i(\gamma_{\ell}(x_1)),\rho)\\
>g_i(\gamma_{\ell}(x_1))-x_1^{\ell+1}C_\ell>0;
$$
hence, $x\in\{g_1>0,\ldots,g_r>0\}\cap M\subset V\subset\widehat{g}^{-1}([\tfrac{c}{2},+\infty))$.

Now we check $U_{\ell,k}\cap\partial M\subset\widehat{g}^{-1}([\tfrac{c}{2},+\infty))$. Since $U_{\ell,k}$ is open in $\betaa M$ and $\widehat{\partial}M$ is dense in $\partial M$ (see Lemma \ref{deltam}), it is sufficient to show that $U_{\ell,k}\cap\widehat{\partial}M$ is contained in $\widehat{g}^{-1}([\tfrac{c}{2},+\infty))$. To that end it is enough to prove that $\mu\in\{g_1>0,\ldots,g_r>0\}_{F_1}$ for all formal paths $\mu\in (U_{\ell,k}\cap M)_{F_1}$. Indeed, $\mu_1(\t)>0$ because $\mu\in M_{F_1}$ and after reparametrizing we may assume $\mu_1(\t)=\t^p$ for some $p\geq 1$. Since $\mu\in(U_{\ell,k})_{F_1}$, we get
$$
\|\mu(\t)-\gamma_\ell(\t)\|^2=(\mu_2(\t)-\gamma_{2\ell}(\t^p))^2+\cdots+(\mu_m(\t)-\gamma_{m\ell}(\t^p))^2<\t^{2(\ell+1)p}
$$
and so $\|\mu(\t)-\gamma_{\ell}(\t^p)\|^2\in(\t)^{2(\ell+1)p}$; hence, since $\|\alpha(\t)-\gamma_{\ell}(\t)\|^2\in(\t)^{2(\ell+1)}$, we deduce $\|\mu(\t)-\alpha(\t^p)\|^2\in(\t)^{2(\ell+1)p}$ and therefore by \ref{gpn3}.\ref{piccolo} $g_i(\mu(\t))>0$ for each index $i=1,\ldots,r$, that is, $\mu\in\{g_1>0,\ldots,g_r>0\}_{F_1}$.

We conclude $U_{\ell,k}=(U_{\ell,k}\cap M)\cup(U_{\ell,k}\cap\partial M)\subset\widehat{g}^{-1}([\tfrac{c}{2},+\infty))\subset{\mathcal D}_{\betaa M}(g)$, as required.\setcounter{substep}{0}
\end{proof}

We finish this section by proving that the zero set of the function $F:=\widehat{f}|_{\partial M}$ induced in $\partial M$ by a function in $f\in{\mathcal S}^*(M)$ is regularly closed. Namely,

\begin{cor}\label{neigh0}
Let $h\in{\mathcal S}^*(M)$, $\widehat{h}:\betaa M\to\R$ be its unique continuous extension and $H:=\widehat{h}|_{\partial M}:\partial M\to\R$. Then
\begin{itemize}
\item[(i)] The set $Z_{\partial M}(H)$ is a closed neighborhood of each free maximal ideal $\gtm_\alpha^*\in(\widehat{\partial}M\setminus\cl_{\betaa M}(\rho_1(M)))\cap Z_{\partial M}(H)$ in $\partial M$. 
\item[(ii)] $Z_{\partial M}(H)=\cl_{\partial M}\big(\Int_{\partial M}(Z_{\partial M}(H))\big)$.
\end{itemize}
\end{cor}
\begin{proof}
(i) Consider the map $\betaa\,{\tt j}:\betaa M_{\lc}\to\betaa M$ induced by the inclusion ${\tt j}:M_{\lc}\hookrightarrow M$. Recall that if $Y:=\rho_1(M)$, then by \cite[6.7(ii)]{fg3} the restriction 
$$
\betaa\,{\tt j}|:\betaa M_{\lc}\setminus(\betaa\,{\tt j})^{-1}(\cl_{\betaa M}(Y))\to\betaa M\setminus\cl_{\betaa M}(Y)
$$ 
is a homeomorphism. In fact, one can check
$$
\widehat{\partial}M\setminus\cl_{\betaa M}(Y)=\betaa\,{\tt j}\big(\widehat{\partial}M_{\lc}\setminus(\betaa\,{\tt j})^{-1}(\cl_{\betaa M}(Y))\big).
$$

Let $\widehat{h\circ{\tt j}}=\widehat{h}\circ(\betaa\,{\tt j}):\betaa M_{\lc}\to\R$ be the (unique) continuous extension of $h\circ{\tt j}$ to $\betaa M_{\lc}$ and consider its restriction 
$$
\widehat{h\circ{\tt j}}|_{\partial M_{\lc}}=\widehat{h}\circ(\betaa\,{\tt j})|_{\partial M_{\lc}}:\partial M_{\lc}\to\R. 
$$
By Corollary \ref{ppp} and \cite[6.1]{fe1} we deduce that $Z_{\partial M_{\lc}}(\widehat{h\circ{\tt j}}|_{\partial M_{\lc}})$ is a neighborhood of $\gtn_\alpha^*:=(\betaa\,{\tt j})^{-1}(\gtm^*_\alpha)$ in $\partial M_{\lc}$. Moreover, $\gtn_\alpha^*\not\in(\betaa\,{\tt j})^{-1}(\cl_{\betaa M}(Y))$ because $\gtm^*_\alpha\not\in\cl_{\betaa M}(Y)$. Therefore $Z_{\partial M_{\lc}}(\widehat{h}\circ(\betaa\,{\tt j})|_{\partial M_{\lc}})\setminus(\betaa\,{\tt j})^{-1}(\cl_{\betaa M}(Y))$ is a neighborhood of $\gtn_\alpha^*$ in $\partial M_{\lc}\setminus(\betaa\,{\tt j})^{-1}(\cl_{\betaa M}(Y))$. Taking images under $\betaa\,{\tt j}$, we conclude: \em $Z_{\partial M}(H)$ is a closed neighborhood of $\gtm_\alpha^*$ in $\partial M$\em.

(ii) We prove the non-obvious inclusion in (ii). Let $\gtm^*\in Z_{\partial M}(H)$ and $g\in{\mathcal S}^*(M)$ be such that $\gtm^*\in \di_{\betaa M}(g)$. We must prove that $\di_{\betaa M}(g)$ meets $\Int_{\partial M}(Z_{\partial M}(H))$. By \cite[4.10]{fg5} there exists $b\in{\mathcal S}^*(M)$ whose continuous extension $\widehat{b}$ to $\betaa M$ satisfies $\widehat{b}(\gtm^*)=1$ and $\widehat{b}|_{\cl_{\betaa M}(Y)}=0$. Consider the continuous extensions $\widehat{h}:\betaa M\to\R$ and $\widehat{g}:\betaa M\to\R$ of $h$ and $g$. By Lemma \ref{deltam}(i) there exists $\gtm_\alpha^*\in\widetilde{\partial}M$ such that 
$$
\widehat{h}(\gtm_\alpha^*)=\widehat{h}(\gtm^*)=0,\ 
\widehat{g}(\gtm_\alpha^*)=\widehat{g}(\gtm^*)\neq 0,\ 
\widehat{b}(\gtm_\alpha^*)=\widehat{b}(\gtm^*)= 1. 
$$
Consequently, $\gtm_\alpha^*\in\widetilde{\partial}M\setminus\cl_{\betaa M}(Y)\subset\widehat{\partial}M\setminus\cl_{\betaa M}(Y)$ and $\gtm_\alpha^*\in\di_{\betaa M}(g)\cap Z_{\partial M}(H)$. Using (i), this implies $\gtm_\alpha^*\in\Int_{\partial M}(Z_{\partial M}(H))$ and we are done. 
\end{proof}

\section{Homeomorphisms between maximal spectra}\label{s5}

In this section we prove Theorem \ref{betabound}. Since both maximal spectra $\betas M$ and $\betaa M$ are homeomorphic (see \ref{homeo}) and our results are of topological nature, we deal only with $\betaa M$. The essential reason for this choice is that each function $f\in {\mathcal S}^*(M)$ admits a unique continuous extension $\widehat{f}:\betaa M\to\R$ (see \ref{cocr}). Before proving Theorem \ref{betabound}, we need some preparation.

\begin{lem}\label{mn}
Let $C$ be the union of bricks of $M$ of dimension $\leq 1$. The set of points of $\betaa M$, which have a metrizable neighborhood in $\betaa M$, equals $M_{\lc}\cup(\cl_{\betaa M}(C)\setminus C)$. Moreover, the points of $\eta(M)\cup(\cl_{\betaa M}(C)\setminus C)$ are those having an open neighborhood in $\betaa M$ that is homeomorphic to the interval $[0,1)$.
\end{lem}
\begin{proof}
Let $T$ be the set of points of $\betaa M$ that have a metrizable neighborhood in $\betaa M$ and let $q\in M_{\lc}\cup(\cl_{\betaa M}(C)\setminus C)$. If $q\in(\cl_{\betaa M}(C)\setminus C)$, then $q$ has an open neighborhood in $\betaa M$ that is homeomorphic to $[0,1)$, which is a metrizable space. On the other hand, $M_{\lc}$ is open in $\betaa M$ and a metrizable neighborhood of all its points. Therefore $M_{\lc}\cup(\cl_{\betaa M}(C)\setminus C)\subset T$. Let us prove $T\subset M_{\lc}\cup(\cl_{\betaa M}(C)\setminus C)$ next. 

Indeed, let $p\in T$ and $W$ be a metrizable neighborhood of $p$ in $\betaa M$. Let $f\in{\mathcal S}^*(M)$ be such that $p\in{\mathcal D}_{\betaa M}(f)\subset W$. We may assume that the unique continuous extension $\widehat{f}:\betaa M\to\R$ of $f$ to $\betaa M$ satisfies $\widehat{f}(p)=c>0$ and we consider the closed semialgebraic subset $Z:=f^{-1}([\tfrac{c}{2},+\infty))=\widehat{f}^{-1}([\tfrac{c}{2},+\infty))\cap M$ of $M$. 

By \ref{closedbeta1}(i) $\betaa Z$ is homeomorphic to $\cl_{\betaa M}(Z)$. It contains $p$ and is metrizable because it is a subset of $W$. Hence, by \cite[5.17]{fg5} the subset $Z^{\geq 2}$ of points of local dimension $\geq2$ of $Z$ is compact. We write $Z=Z^{\geq 2}\cup C_0$ where $C_0$ is the union of bricks of $Z$ of dimension $\leq 1$ (see Proposition \ref{bricks}), so it is a closed subset of $M$. By \ref{closedbeta1}(i) we can identify $\betaa Z=\betaa Z^{\geq 2}\cup\betaa C_0=Z^{\geq2}\cup\betaa C_0$. We distinguish two cases:

\vspace{1mm}
\begin{substeps}{mn}
If $p\in Z^{\geq 2}$, then $p\in M$ and $f(p)=c$. Thus, since $Z^{\geq2}$ is compact, $Z$ is by Remark \ref{curvas}(iii) a locally compact neighborhood of $p$ in $M$; hence, by Proposition \ref{rho} $p\in M_{\lc}$. In fact, $p\in M_{\lc}\setminus\eta(M)$.
\end{substeps}

\vspace{1mm}
\begin{substeps}{mn}
If $p\not\in Z^{\geq 2}$, then $p\in\betaa C_0\setminus Z^{\geq 2}$. By Lemma \ref{fc} $\cl_{\betaa M}(Z)\cong\betaa Z$ is a neighborhood of $p$ in $\betaa M$. Thus, $\betaa C_0\setminus Z^{\geq 2}$ is a neighborhood of $p$ in $M$ and there are two possible situations:

(a) $p\in C_0$ and so $C_0\setminus Z^{\geq 2}=(\betaa C_0\setminus Z^{\geq 2})\cap M$ is a locally compact neighborhood of $p$ in $M$; hence, $p\in M_{\lc}$.

(b) $p\in\betaa C_0\setminus C_0\subset\betaa M\setminus M$; hence, $\betaa C_0\setminus Z^{\geq 2}$ is a neighborhood of $p$ in $\betaa M$. By \cite[2.9.10]{bcr} the semialgebraic curve $C_0$ is the disjoint union of a finite set ${\mathcal F}:=\{q_1,\ldots,q_r\}$ and finitely many Nash submanifolds $N_i$ for $i=1,\ldots,s$ where each of them is Nash diffeomorphic to the open interval $(0,1)$. Thus, $\betaa C_0={\mathcal F}\cup\bigcup_{i=1}^s\cl_{\betaa C_0}(N_i)$ and we claim that there exists exactly one index $i=1,\ldots,s$ such that $p\in\cl_{\betaa C_0}(N_i)$. Otherwise there exist $1\leq i<j\leq s$ such that 
$$
p\in\cl_{\betaa C_0}(N_i)\cap\cl_{\betaa C_0}(N_j)=\cl_{\betaa C_0}(\cl_M(N_i)\cap\cl_M(N_j))
$$ 
and it must be one of the points $q_k\in {\mathcal F}\subset M$ as this last intersection is non empty (see \ref{closedbeta1}). This is a contradiction because $p\in\betaa C_0\setminus C_0=\betaa C_0\setminus M$. Since $p\not\in M$, we may assume that $p\in\cl_{\betaa C_0}(\cl_M(N_1))\setminus({\mathcal F}\cup\bigcup_{i=2}^s\cl_{\betaa C_0}(N_i))$. Moreover, $\cl_M(N_1)=\cl_{C_0}(N_1)$ is homeomorphic either to $(0,1)$ or to $[0,1)$. Otherwise $\cl_M(N_1)$ is homeomorphic either to the unit circle $\sph^1$ or to $[0,1]$ and it is compact; hence, by \ref{freefixed}
$$
p\in\cl_{\betaa C_0}(N_1)=\cl_{\betaa C_0}(\cl_M(N_1))=\cl_M(N_1)\subset M,
$$
which is a contradiction. Now we obtain by \ref{closedbeta1}(i) and \cite[4.9]{fg5}
$$
\cl_{\betaa C_0}(N_1)=\cl_{\betaa C_0}(\cl_{C_0}(N_1))\cong\betaa\cl_{C_0}(N_1)\cong\betaa\,[0,1)\cong\betaa\,(0,1)\cong[0,1].
$$
On the other hand, $\cl_{\betaa C_0}(N_1)$ is a neighborhood of $p$ in $\betaa M$ because
$$
\betaa M\setminus\Big(Z^{\geq2}\cup {\mathcal F}\cup\bigcup_{i=2}^s\cl_{\betaa C_0}(N_i)\Big)=\cl_{\betaa C_0}(N_1)\setminus\Big(Z^{\geq2}\cup {\mathcal F}\cup\bigcup_{i=2}^s\cl_{\betaa C_0}(N_i)\Big)
$$
is a neigborhood of $p$ in $\betaa M$. As $\cl_{\betaa C_0}(N_1)\cong[0,1]$, $N_1$ is homeomorphic to $(0,1)$ and $p\in\cl_{\betaa C_0}(N_1)\setminus N_1$, we conclude that $\cl_{\betaa C_0}(N_1)$ is a neighborhood of $p$ in $\betaa M$ that is homeomorphic to $[0,1)$ and whose intersection with $M$ is homeomorphic to $(0,1)$. Therefore $p\in\cl_{\betaa M}(C)\setminus C$, as required.
\end{substeps}\setcounter{substep}{0} 
\end{proof}

Now we are ready to prove Theorem \ref{betabound}.
\begin{proof}[Proof of Theorem \em\ref{betabound}]
By symmetry all is reduced to show $\gamma(N_{\lc}\setminus\eta(N_{\lc}))\subset M_{\lc}\setminus\eta(M_{\lc})$. Notice that $\eta(N)=\eta(N_{\lc})$ and $\eta(M)=\eta(M_{\lc})$ because the semialgebraic curves are by Remark \ref{curvas} locally compact. Thus, it is sufficient to check $\gamma(N_{\lc}\setminus\eta(N))\subset M_{\lc}\setminus\eta(M)$.

Indeed, let $p\in N_{\lc}\setminus\eta(N)$ and $q:=\gamma(p)$. As $N_{\lc}$ is open in $\betaa N$, there exists a compact semialgebraic neighborhood $K\subset N_{\lc}$ of $p$ in $\betaa N$, which is clearly metrizable. Thus, $\gamma(K)$ is a metrizable neighborhood of $q$ in $\betaa M$. By Lemma \ref{mn} either $q\in M_{\lc}\setminus \eta(M)$ or it has an open neighborhood that is homeomorphic to the interval $[0,1)$. The latter would mean that $p$ has a neighborhood in $N$ that is homeomorphic to the interval $[0,1)$, and so $p\in\eta(N)$, which is a contradiction. Thus, $\gamma(p)=q\in M_{\lc}\setminus\eta(M)$, as wanted.
\end{proof}

We now show the relevance of the semialgebraic character of maps between semialgebraic sets in the study of properties of the operator $\betaa$ .
\begin{examples}\label{gpn}
(i) Let $M:=\R^{2}\setminus\{0\}$ and consider the smooth path $\gamma:(0,1]\to M,\ t\mapsto(t,\lambda\exp(-1/t))$ where $\lambda$ is a fixed positive real number. Then 
\begin{itemize}
\item[(1)] For all $f\in{\mathcal S}^*(M)$ there exists the limit $\lim_{t\to0}\,(f\circ\gamma)(t)\in\R$.
\item[(2)] The set $\gtm^*:=\{f\in{\mathcal S}^*(M):\, \lim_{t\to0}\,(f\circ\gamma)(t)=0\}$ is a maximal ideal of ${\mathcal S}^*(M)$.
\item[(3)] $\gtm^*=\gtm^*_\alpha$ where $\alpha:(0,1]\to M,\ t\mapsto(t,0)$.
\end{itemize}
\begin{proof}
Let $f\in{\mathcal S}^*(M)$ and $\widehat{f}:\betaa M\to\R$ be the unique continuous extension of $f$ to $\betaa M$. Assume $\widehat{f}(\gtm^*_\alpha)=0$ and observe that statements (1), (2) and (3) are straightforward consequences of the following equality.

\vspace{1mm}
\begin{substeps}{gpn}\label{evaluando} $\lim_{t\to0^+}f(t,\lambda\exp(-1/t))=0=\widehat{f}(\gtm^*_\alpha)=\lim_{t\to0^+}f(t,0)=0$.
\end{substeps}

\vspace{1mm}
Indeed, $Z_{\partial M}(\widehat{f})$ is a closed neighborhood of $\gtm_\alpha^*$ in $\partial M=\betaa M\setminus M$ by Corollary \ref{neigh0}. Thus, there exists $g\in{\mathcal S}^*(M)$ such that $\gtm_{\alpha}^*\in{\mathcal D}_{\betaa M}(g)\cap\partial M\subset Z_{\partial M}(\widehat{f})$. We may also assume $c:=\widehat{g}(\gtm_\alpha^*)>0$. Consider the closed semialgebraic set $Z:=g^{-1}([\tfrac{c}{2},+\infty))\cap\{x^2+y^2\leq1\}$. Since $\gtm^*_\alpha\in{\mathcal D}_{\betaa M}(g)\cap\partial M$, there exists $\veps>0$ such that $Y_\veps:=(0,\veps]\times\{0\}\subset Z$. 

Otherwise, as $Z$ is semialgebraic, there exists $\veps>0$ such that the closed semialgebraic subsets $Z$ and $Y=Y_\veps$ of $M$ are disjoint. Then there exists by \cite{dk} a semialgebraic function $h\in{\mathcal S}^*(M)$ such that $h|_Z=0$ and $h|_Y=1$. Thus, $\widehat{h}(\gtm_\alpha^*)=1$ and by \ref{closedbeta1}(i)\&(ii) and Lemma \ref{fc} we obtain
$$
\gtm^*_\alpha\not\in {\mathcal Z}_{\betaa M}(h)\supset\cl_{\betaa M}(Z)=\cl_{\betaa M}(\widehat{g}^{-1}([\tfrac{c}{2},+\infty))\cap\cl_{\betaa M}(x^2+y^2\leq 1),
$$
which is a contradiction.

Since the Taylor series at the origin of the function $\lambda\exp(-1/t)$ is identically zero, the image of $\gamma|:(0,\delta]\to M$ for $\delta>0$ small enough is contained in $Z$. On the other hand, since $\widehat{g}^{-1}([\tfrac{c}{2},+\infty))\cap\partial M\subset Z_{\partial M}(\widehat{f})$, the closure of the graph $T$ of $f|_Z$ in $\R^3$ is $T\cup\{(0,0,0)\}$; this implies \ref{gpn}.\ref{evaluando}, as wanted.
\end{proof}

(ii) Consider the homeomorphism $\varphi:\R^2\to\R^2$ given by the formulas
$$
(x,y)\mapsto
\left\{\begin{array}{ll}
\!\!\big(x,\big(1-\frac{\exp(-1/x)}{x}\big)(2y-x)+\exp(-1/x)\big)&\text{ if $0\leq\frac{1}{2}x\leq y\leq x$,}\\[4pt]
\!\!\big(x,\frac{2\exp(-1/x)}{x}y\big)&\text{ if $0\leq y \leq\frac{1}{2}x$,}\\[5pt]
\!\!(x,y)&\text{ if $y\leq 0$ or $x\leq y$.}
\end{array}
\right.
$$
Write $M:=\R^2\setminus\{(0,0)\}$. Since $\varphi(0,0)=(0,0)$, the restriction $\psi:=\varphi|_{M}:M\to M$ is a homeomorphism. Note that $\varphi(t,\lambda t)=(t,\lambda\exp(-1/t))$ for $0\leq\lambda\leq 1/2$ and $t>0$. Thus, the homeomorphism $\psi:M\to M$ cannot be extended to a homeomorphism $\widehat{\psi}:\betaa M\to\betaa M$ because, in view of Example \ref{gpn}(i), such an extension would map the (distinct) maximal ideals $\gtm_{\lambda}^*:=\{f\in{\mathcal S}^*(M):\, \lim_{t\to0^+}f(t,\lambda t)=0\}$, where $0\leq\lambda\leq 1/2$, onto the maximal ideal $\gtm_0^*$.

The behavior of a non-semialgebraic homeomorphism between semialgebraic sets turns out to be impredictable with respect to its possible extensions to the semialgebraic Stone--\v{C}ech compactification. In fact, semialgebraic paths become useless in the absence of semi\-algebraicity.\fina
\end{examples}

\appendix
\section{Differences between the sets $\widetilde{\partial}M$, $\widehat{\partial}M$ and $\partial M$}\label{diferencias}

The purpose of this Appendix is to prove that the non-empty differences $\partial M\setminus\widehat{\partial}M$ and $\widehat{\partial}M\setminus\widetilde{\partial}M$ are respectively dense in $\partial M$ and $\widehat{\partial}M$ under mild conditions. Recall that $M^{\geq2}$ denotes the (semialgebraic) subset of points of $M$ of local dimension $\geq 2$.

\begin{prop}\label{diferencias-s}
Assume that $M=M^{\geq2}$ is not compact. Then $\partial M\setminus\widehat{\partial}M$ is dense in $\partial M$ and $\widehat{\partial}M\setminus\widetilde{\partial}M$ is dense in $\widehat{\partial}M$.
\end{prop}

We begin with some preliminary results.

\begin{lem}\label{pppoints0}
Assume that $M$ is bounded. Then $\partial M\setminus\widehat{\partial}M\neq\varnothing$ if and only if $M^{\geq2}$ is not compact. Moreover, if $M^{\geq2}$ is compact, then $\widetilde{\partial} M=\partial M$ is a finite set.
\end{lem}
\begin{proof}
Suppose first that $M^{\geq2}$ is compact. The finiteness of $\partial M$ follows from \cite[5.17]{fg5}; hence, by Lemma \ref{deltam}(iii) $\widetilde{\partial}M=\partial M$.

Conversely, suppose that $M^{\geq2}$ is not compact. By \cite[7.1(i)]{fe1} there exists a point $p\in\cl_{\R^m}(M^{\geq2})\setminus(\cl_{\R^m}(\rho_1(M^{\geq2}))\cup M^{\geq2})$. Let $C:=\cl_M(M\setminus M^{\geq2})$, for which $\rho_1(C)=\varnothing$ by Remark \ref{curvas}; hence, $\rho_1(M)=\rho_1(M^{\geq2})$. Moreover, $M^{\geq 2}$ is closed in $M$ and so $p\in\cl_{\R^m}(M)\setminus(\cl_{\R^m}(\rho_1(M))\cup M)$ and $\dim_p(\cl_{\R^m}(M))\geq2$. By \cite[7.1(ii)]{fe1} there exists a maximal ideal $\gtm_1^*$ of ${\mathcal S}^*(M)$ of height $\geq 2$ such that $\hgt(\gtm_1)=0$. This implies by Corollary \ref{ppp} that $\gtm_1^*\in\partial M\setminus\widehat{\partial}M$, as required.
\end{proof}

\begin{lem}[Behavior of the operators $\widetilde{\partial}$ and $\widehat{\partial}$]\label{delrel}
Assume that $M$ is bounded and let $Y\subset M$ be a closed semialgebraic subset of $M$. Let $C$ be the closure of the set of points of $M$ of local dimension $\leq 1$ in $M$. Since the semialgebraic sets $Y$, $M^{\geq 2}$ and $C$ are closed in $M$, we identify 
$$
\cl_{\betaa M}(Y)\equiv\betaa Y,\quad\cl_{\betaa M}(M^{\geq2})\equiv\betaa M^{\geq 2}\quad\text{and}\quad\cl_{\betaa M}(C)\equiv\betaa C.
$$ 
Then
\begin{itemize}
\item[(i)] $\widetilde{\partial} Y=\widetilde{\partial}M\cap\partial Y$ and $\widehat{\partial} Y=\widehat{\partial}M\cap\partial Y$.
\item[(ii)] $\partial M=\partial M^{\geq2}\sqcup\partial C$. 
\item[(iii)] $\widetilde{\partial} M=\widetilde{\partial}M^{\geq 2}\sqcup\widetilde{\partial} C$ and $\widehat{\partial} M=\widehat{\partial}M^{\geq 2}\sqcup\widehat{\partial} C$.
\end{itemize}
\end{lem}
\begin{proof}
(i) Let us check first $\widehat{\partial} Y=\widehat{\partial}M\cap\partial Y$. For the non-obvious inclusion let $\gtm^*_\alpha\in\widehat{\partial}M\cap\partial Y$. Suppose by contradiction that $\gtm^*_\alpha\not\in\widehat{\partial} Y$, that is, $\alpha\not\in Y_{F_1}$. Thus, $\alpha\in(M\setminus Y)_{F_1}$ and there exists $g\in{\mathcal S}^*(M)$ such that $\alpha\in(D_M(g))_{F_1}\subset(M\setminus Y)_{F_1}$. In particular, $g|_{Y}\equiv 0$ and $\psi_{\alpha}(g)\neq0$. Write $\psi_{\alpha}(g):=a\t^p+\cdots$ for some $a\neq 0$ and a non-negative rational number $p$ and $\|\alpha(\t)-\alpha(0)\|:=b\t^q+\cdots$ for some $b\neq 0$ and a positive $q\in\Q$. Recall that $\alpha(0)\not\in M$ because $\gtm_\alpha^*\in\partial Y$ and consider the bounded semialgebraic function
$$
f:M\to\R,\ x\mapsto\frac{g^2(x)}{g^2(x)+\|x-\alpha(0)\|^{2(p/q)+1}},
$$
which vanishes identically on $Y$ and satisfies $\psi_{\alpha}(f)(0)=1$. Thus, $f\in\ker\psi\setminus\gtm^*_\alpha$ where $\psi:{\mathcal S}^*(M)\to{\mathcal S}^*(Y),\,h\mapsto h|_Y$. This contradicts the fact that $\gtm^*_\alpha\in\partial Y\equiv\cl_{\betaa M}(Y)\setminus Y$ because $\cl_{\betaa M}(Y)$ is the collection of those maximal ideals of ${\mathcal S}^*(M)$ containing $\ker\psi$ by \cite[6.3]{fg3}. The first equality in (i) follows from the equality already proved above after taking into account that the semialgebraic character of a formal path does not depend on the semialgebraic set where it is considered.

Statement (ii) follows by considering the connected components of $\partial M$ and noticing that the union of the ones that are singletons belongs to $\partial C$. Statement (iii) follows easily from (ii).
\end{proof}

\begin{remark}
The assumption $M=M^{\geq2}$ in Proposition \ref{diferencias-s} is not restrictive. Let $C$ be the closure of the set of points of $M$ of local dimension $\leq 1$ in $M$. Since $\dim(C)\leq 1$, we deduce from Lemma \ref{pppoints0} that $\widetilde{\partial}C=\widehat{\partial}C=\partial C$. By Lemma \ref{delrel} we obtain
$$
\partial M\setminus\widehat{\partial}M=\partial M^{\geq2}\setminus\widehat{\partial}M^{\geq2}\quad\text{and}\quad\widehat{\partial}M\setminus \widetilde{\partial}M=\widehat{\partial}M^{\geq2}\setminus \widetilde{\partial}M^{\geq2},
$$
so Proposition \ref{diferencias-s} is conclusive.
\end{remark}

\begin{proof}[Proof of Proposition \em\ref{diferencias-s}]
Our aim is to reduce the problem to prove: 

\vspace{1mm}\setcounter{substep}{0}
\begin{substeps}{diferencias-s}\label{redtrianT}
\em The sets $\partial T\setminus\widehat{\partial}T$ and $\widehat{\partial}T\setminus \widetilde{\partial}T$ are not empty for the punctured triangle 
\begin{center}
$T:=\{(x,y)\in\R^2:\,0\leq y\leq x\leq1\}\setminus\{(0,0)\}.$
\end{center}
\end{substeps}

Suppose \ref{diferencias-s}.\ref{redtrianT} is already proved and let us show the statement under the assumption that $M$ is bounded. Let $f\in {\mathcal S}^*(M)$ be such that $\di_{\betaa M}(f)$ meets $\partial M$ and $\widehat{f}:\betaa M\to\R$ be the (unique) continuous extension of $f$ to $\betaa M$. Since $\widetilde{\partial}M$ is dense in $\partial M$ (see Lemma \ref{deltam}(iii)), $\widetilde{\partial}M$ meets $\di_{\betaa M}(f)$, too. Let $\gtm_\alpha^*\in\widetilde{\partial}M\cap\di_{\betaa M}(f)$ and $c:=\widehat{f}(\gtm^*_\alpha)\neq0$ that we assume to be $>0$. Thus, $\gtm^*_\alpha\in{\mathcal D}_{\betaa M}(f-\tfrac{c}{2}+|f-\tfrac{c}{2}|)$. Substituting $M$ by $\gr(f)$, we may assume that $f$ can be extended continuously to $X:=\cl_{\R^m}(M)$; denote such extension with $f$. By \cite[4.3\&4.6]{fg5} there exists a continuous surjective map $\rho:\betaa M\to X$, which is the identity on $M$, and $\rho(\partial M)=X\setminus M$. Moreover, $\widehat{f}=f\circ\rho$ and $p:=\rho(\gtm_\alpha^*)\in X\setminus M$ satisfies $f(p)=c$. 

Define $Y_0:=\{p\}$, $Y_1:=\{f-\tfrac{c}{2}>0\}$ and $Y_2:=M\cap Y_1$. By \cite[9.2.1]{bcr} there exists a finite simplicial complex $K$ and a semialgebraic homeomorphism $\Phi:|K|\to X$ such that each semialgebraic set $Y_j$ is the union of some $\Phi(\sigma^0)$ where each $\sigma^0$ is the open simplex associated with a simplex $\sigma\in K$. We identify $X$ with $|K|$ and choose a simplex $\tau$ of $K$ of dimension $\geq2$ that has $p$ as a vertex and whose associated open simplex $\tau^0$ is contained in $Y_2$. Let $p_1,p_2\in\tau^0$ be two points that are not colinear with $p$. For the closed triangle $T_1$ with vertices $p,p_1,p_2$ it holds that $T_1\setminus\{p\}\subset\tau\subset Y_2$ is a closed semialgebraic subset of $M$. Moreover, the remainder satisfies $\partial T_1\subset\di_{\betaa M}(f)$. Thus, the differences $\partial T_1\setminus\widehat{\partial}T_1$ and $\widehat{\partial}T_1\setminus\widetilde{\partial}T_1$ are by \ref{diferencias-s}.\ref{redtrianT} not empty and the open set $\di_{\betaa M}(f)$ meets the differences $\partial M\setminus\widehat{\partial}M$ and $\widehat{\partial}M\setminus \widetilde{\partial}M$ by Lemma \ref{delrel}, as required.

Let us prove \ref{diferencias-s}.\ref{redtrianT} next. By Lemma \ref{pppoints0} we obtain $\partial T\setminus\widehat{\partial}T\neq\varnothing$. In order to prove $\widehat{\partial}T\setminus\widetilde{\partial}T\neq\varnothing$ choose the formal series $\alpha_1(\t)=\t$ and $\alpha_2(\t)=\sum_{n\geq2}n!\t^n\in\R[[\t]]\setminus\R[[\t]]_{\rm alg}$ and the formal path $\alpha:=(\alpha_1,\alpha_2)\in\R[[\t]]^2$. Note that $\alpha\in T_{F_1}$ and let us show $\gtm_\alpha^*\in\widehat{\partial}T\setminus\widetilde{\partial}T$. Indeed, for each $k\geq 2$ consider the function $f_k\in{\mathcal S}^*(T)$ given by the formula
$$
f_k(x,y):=\frac{(y-p_k(x))^2}{(y-p_k(x))^2+x^{2k}}\qquad\text{where $p_k(x):=\sum_{n=2}^{k}n!x^n$} .
$$
For each evaluation we have $\psi_{\alpha}(f_k)(0)=0$, so we deduce $f_k\in\gtm^*_\alpha$ for all $k\geq2$ (see \ref{mifp}). Suppose now that $\gtm_\alpha=\gtm_\mu$ for some $\mu\in\R[[\t]]_{\rm alg}^2$ with $\mu\in T_{F_1}$. To obtain a contradiction, it is enough to check that $f_k\not\in\gtm_\mu$ for some $k\geq2$. Without loss of generality and after reparametrizing $\mu$, we may assume $\mu(\t)=(\t^j,\mu_2(\t))$ for some integer $j\geq1$ and some analytic series $\mu_2(\t)\in\R[[\t]]_{\rm alg}$ whose order is $\geq j$. As the series $\alpha_2(\t^j)$ is not analytic, for the difference it holds $\alpha_2(\t^j)-\mu_2(\t)\neq 0$ and its order is $k\geq 1$. Thus, $\psi_\mu(f_k)(0)\neq 0$ and so $f_k\not\in\gtm_\mu$, as wanted.
\end{proof}

\section*{Acknowledgments}
We are indebted to S. Schramm for a careful reading of the final version and for the suggestions to refine its redaction.

\bibliographystyle{amsalpha}

\end{document}